\documentclass[12pt,a4paper,reqno]{amsart}

\title{On characteristic classes of vector bundles\\[5pt] over quantum spheres}
\date{January 2024}
\author[F.~D'Andrea, G.~Landi, C.~Pagani]{Francesco D'Andrea, Giovanni Landi, Chiara Pagani}

\address[F.~D'Andrea]
{Universit\`a di Napoli Federico II and INFN--Napoli, Napoli, Italy}
\email{francesco.dandrea@unina.it}

\address[G.~Landi]
{Universit\`a di Trieste and INFN--Trieste, Trieste, Italy}
\email{landi@units.it}

\address[C.~Pagani]
{Universit\`a di Napoli Federico II and INFN--Napoli, Napoli, Italy}
\email{chiara.pagani@unina.it}

\usepackage[hidelinks,pdfstartview=FitH]{hyperref} 
\usepackage{fullpage}
\usepackage{amssymb,interval,mathrsfs,enumitem}
\usepackage{tikz}
\usepackage[T1]{fontenc}
\usepackage[english]{babel}
\setenumerate{label=\textup{(\roman*)},itemsep=2pt,topsep=2pt,leftmargin=2.5em}
\usetikzlibrary{arrows,cd,arrows.meta,calc}

\usepackage{mathtools}
\mathtoolsset{showonlyrefs}

\allowdisplaybreaks[4]
\newcommand{\doi}[1]{\url{https://doi.org/#1}}
\newcommand{\isbn}[1]{\url{https://isbnsearch.org/isbn/#1}}
\newcommand{\arxiv}[1]{\href{https://arxiv.org/abs/#1}{preprint arXiv:#1}}
\newcommand{\web}[1]{\url{#1}}

\numberwithin{equation}{section}

\newtheorem{prop}{Proposition}[section]
\newtheorem{lemma}[prop]{Lemma}
\newtheorem{cor}[prop]{Corollary}

\theoremstyle{remark}
\newtheorem{ex}[prop]{Example}
\newtheorem{rem}[prop]{Remark}

\usepackage{etoolbox}
\AtEndEnvironment{rem}{\hfill\ensuremath{\square}}

\newcommand{\dott}{{\scriptstyle{\bullet}}\,}

\newcommand{\N}{\mathbb{N}}
\newcommand{\Z}{\mathbb{Z}}
\newcommand{\R}{\mathbb{R}}
\newcommand{\C}{\mathbb{C}}
\newcommand{\ket}[1]{\left|#1\right>}
\newcommand{\tr}{\mathrm{Tr}}
\newcommand{\id}{\mathrm{Id}}
\newcommand{\mat}[2]{\bigg(\!\begin{array}{{#1}{#1}}#2\end{array}\!\bigg)}

\newcommand{\ot}{\otimes}
\newcommand{\uu}[1]{{U_{\scriptscriptstyle{(#1)}}}}

\newcommand{\pp}[1]{{P_{\scriptscriptstyle{(#1)}}}}
\renewcommand{\tt}[1]{{T_{\scriptscriptstyle{(#1)}}}}

\newcommand{\ww}[1]{{W_{\scriptscriptstyle{(#1)}}}}
\newcommand{\qq}[1]{{Q_{\scriptscriptstyle{(#1)}}}}

\newcommand{\half}{\scriptscriptstyle{\frac{1}{2}}}
\newcommand{\qbin}[2]{{\genfrac{[}{]}{0pt}{}{#1}{#2}}}
 
\newcommand{\qqnum}[1]{[\![#1]\!]} 
\newcommand{\beq}{\begin{equation}}
\newcommand{\eeq}{\end{equation}}

\makeatletter
\def\@tocline#1#2#3#4#5#6#7{\relax
  \ifnum #1>\c@tocdepth 
  \else
    \par \addpenalty\@secpenalty\addvspace{#2}%
    \begingroup \hyphenpenalty\@M
    \@ifempty{#4}{%
      \@tempdima\csname r@tocindent\number#1\endcsname\relax
    }{%
      \@tempdima#4\relax
    }%
    \parindent\z@ \leftskip#3\relax \advance\leftskip\@tempdima\relax
    \rightskip\@pnumwidth plus4em \parfillskip-\@pnumwidth
    #5\leavevmode\hskip-\@tempdima
      \ifcase #1
       \or\or \hskip 1em \or \hskip 2em \else \hskip 3em \fi%
      #6 \hskip 0.5em \nobreak\relax
    \dotfill\hbox to\@pnumwidth{\@tocpagenum{#7}}\par
    \nobreak
    \endgroup
  \fi}
\makeatother

\begin{document}

\subjclass[2020]{Primary: 20G42; Secondary: 46L89; 58B34.}

\keywords{Quantum spaces, K-theory, principal and associated vector bundles, instantons on $q$-spheres, Euler class, Chern characters.}

\begin{abstract}
We study the quantization of spaces whose K-theory in the classical limit is the ring of dual numbers $\Z[t]/(t^2)$. For a compact Hausdorff space we recall necessary and sufficient conditions for this to hold. For a compact quantum space, we give sufficient conditions that guarantee there is a morphism of abelian groups
$K_0 \to \Z[t]/(t^2)$ compatible with the tensor product of bimodules.
Applications include the standard Podle\'s sphere $S^2_q$ and a quantum $4$-sphere $S^4_q$ coming from quantum symplectic groups. For the latter, the K-theory is generated by the Euler class of the instanton bundle.
We give explicit formulas for the projections of vector bundles on $S^4_q$ associated to the principal $SU_q(2)$-bundle $S^7_q \to S^4_q$ via irreducible corepresentations of $SU_q(2)$, and compute their characteristic classes.
\end{abstract}

\maketitle

\begin{center}
\begin{minipage}{0.8\textwidth}
\parskip=0pt\small\tableofcontents
\end{minipage}
\end{center}

\parskip = 1 ex

\bigskip

\section{Introduction}\label{sec:1}
There are many well-known techniques to compute the K-theory groups of a topological space or of a C*-algebra. To mention two main examples, one has the six-term exact sequence associated to an extension of C*-algebras, and the six-term exact sequence associated to a one-surjective pullback diagram of C*-algebras (or one-injective pushout diagram of topological spaces). The latter is usually called \emph{Mayer-Vietoris sequence} in K-theory, by analogy with the celebrated long exact sequence in algebraic topology \cite{Hil86}.
One drawback of these exact sequences is that the connecting homomorphism is only a morphism of abelian groups, even for topological spaces where the K-theory is actually a (commutative) ring.

Computing explicitly the ring structure of $K^*(X)$, for a topological space $X$, is usually harder (where ``explicitly'' means, for example, finding a presentation by generators and relations). Known examples include complex projective spaces and even-dimensional spheres, both of which are computed with suitable ad hoc techniques. For projective spaces, Atiyah and Todd \cite{AT60} proved that $K^0(\C P^n)$ is the ring $\Z[t]/(t^{n+1})$ of truncated polynomials in $t:=1-[L_1]$, $L_1$ being the canonical line bundle. The K-ring of an even-dimensional sphere is obtained as a corollary of Bott's periodicity theorem \cite[Thm.~5.5]{Hus94}, and given by the ring of dual numbers $K^0(S^{2n})\cong\Z[t]/(t^2)$ (in all these cases, $K^1$ is zero).

A more general approach, when $X\cong G/H$ is a homogeneous space, is given by Hodgkin \cite{Hod75}.
If $G$ is a compact Lie group and $H$ a closed subgroup, there is a homomorphism between representation rings $R(G)\to R(H)$, given by the restriction of representations, which is injective if $H$ has maximal rank. The construction of vector bundles associated to the principal $H$-bundle $G\to X$ gives a ring homomorphism $R(H)\to K^0(X)$. Since 
bundles associated to representations of $G$ are trivial, the ring $R(G)$ acts on $K^0(X)$ simply by multiplication by dimension, and we have a natural ring homomorphism
\begin{equation}\label{eq:Hodgkin}
R(H)\otimes_{R(G)}\Z\to K^0(X) .
\end{equation}
Hodgkin proved that, if $H$ has maximal rank and $\pi_1(G)$ is torsion-free, then
\eqref{eq:Hodgkin} is an isomorphism \cite[Cor.~page 71]{Hod75} (under the same assumptions, $K^1(X)=0$).
Hodgkin's theorem applies to both   $\C P^n\cong SU(n+1)/U(n)$ and  $S^{2n}\cong Spin(2n+1)/Spin(2n)$, for all $n\geq 1$.

Given the above discussion, already in the classical case of compact Hausdorff spaces, it is interesting to have simple conditions that, from the knowledge of $K^0(X)$ as an abelian group, allow one to determine its ring structure. It is in this spirit that, in the first part of the paper, we prove the following simple result.

\begingroup
\renewcommand\theprop{\protect{\ref{prop:next}}}
\begin{prop}
Let $X$ be a $2n$-dimensional oriented closed manifold, $n\geq 1$. 
Assume that $K^0(X)\cong\Z^2$ as an abelian group, and that there is a complex vector bundle $E\to X$ with rank $r$ and Chern number $c=\pm 1$. Then,
$K^0(X)=\Z[t]/(t^2)$ with $t=r-[E]$.
\end{prop}
\endgroup

We also observe (cf.~Proposition \ref{prop:before}) that the assumptions of the previous proposition are equivalent to $X$ being a rational homology sphere with torsion-free K-groups.

Examples include $S^2$ and $S^4$. On $S^2$, $E$ is the vector bundle associated to the 1st Hopf fibration $S^3\to S^2$ and to the fundamental representation of $U(1)$. On $S^4$, $E$ is the vector bundle associated to the 2nd Hopf fibration $S^7\to S^4$ and to the fundamental representation of $SU(2)$. In physical language, the K-theory ring of $S^2$ resp.~$S^4$ is generated by the basic Dirac monopole resp.~instanton bundle 
of the Yang-Mills theory. In both cases, $t=r-[E]$ is the Euler characteristic (in K-theory) of the vector bundle $E$.
The same  computation goes through for any even-dimensional sphere (see Proposition \ref{prop:sfereclassiche}).

Next, we want to find a result for compact quantum spaces which is as close as possible to Proposition \ref{prop:next}.
The K-theory of a C*-algebra has no canonical ring structure but,  by expanding the binomial  we can rephrase the relation $t^2=(r-[E])^2=0$ as
\begin{equation}\label{eq:makesense}
r^2-2r[E]+[E\otimes E]=0 .
\end{equation}
Phrased like this, the result can now be dualized and then generalized, in a suitable way, to noncommutative spaces.

In the quantum world, a manifold is replaced by an associative unital *-algebra $\mathscr{B}$, with a \mbox{C*-enveloping} algebra $B$ encoding the topology. Vector bundles are replaced by $\mathscr{B}$-bimodules that are projective and finitely generated as right (say) $\mathscr{B}$-modules. In this way, we can talk about their class in $K_0(B)$ (they have associated a projection in $M_\infty(\mathscr{B})\subset M_\infty(B)$), and we can also take tensor products over $\mathscr{B}$ and make sense of \eqref{eq:makesense}.
Our main source of examples are vector bundles associated to principal bundles, and the property that we need is that their pullback to the total space of the bundle is trivial. We shall then assume that $\mathscr{B}$ is a unital *-subalgebra of some other *-algebra $\mathscr{A}$ and focus on those bimodules that ``trivialize'' over $\mathscr{A}$. We denote by $\mathrm{Vect}_{\mathscr{A}}(\mathscr{B})$ the family of isomorphism classes of such bimodules.
This is a framework of Propositions \ref{prop:undertheassumption} and \ref{prop:2q}, which we partially restate below.

\begingroup
\renewcommand\theprop{\protect{\ref{prop:undertheassumption}(i-ii)}}
\begin{prop}
Assume that there is a character $\varepsilon:\mathscr{A}\to\C$ and
a bounded unital *-representation $\pi:\mathscr{A}\to\mathcal{B}(\mathcal{H})$ such that
\[
\pi(b)-\varepsilon(b)\in\mathcal{L}^1(\mathcal{H}) \;,\quad\forall\;b\in\mathscr{B}.
\]
Then:
\begin{enumerate}
\item There is a homomorphism of unital semirings
\[
\mathrm{ch}:\big(\mathrm{Vect}_{\mathscr{A}}(\mathscr{B}),\oplus,\otimes\big)  \to\Z[t]/(t^2) ,
\quad\quad
\mathrm{ch}(\mathscr{E})=\mathrm{ch}_0(\mathscr{E})+\mathrm{ch}_1(\mathscr{E})t ,
\]
given, for any $\mathscr{E}\cong p\,\mathscr{B}^N$, by the formulas
\begin{equation*}
\mathrm{ch}_0(\mathscr{E}):=\varepsilon(\tr\,p) , \qquad\quad
\mathrm{ch}_1(\mathscr{E})=\tr_{\mathcal{H}}\big(\pi(\tr\,p)-\varepsilon(\tr\,p)\big) .
\end{equation*}

\item
The underlying morphism of abelian semigroups factors through a homomorphism of abelian groups $K_0(B)\to\Z[t]/(t^2)$.

\end{enumerate}
\end{prop}
\renewcommand\theprop{\protect{\ref{prop:2q}}}
\begin{prop}
Under the assumptions of Proposition \ref{prop:undertheassumption},
if $K_0(B)\cong\Z^2$ and there is $\mathscr{E}\in\mathrm{Vect}_{\mathscr{A}}(\mathscr{B})$ with
$\mathrm{ch}_1(\mathscr{E})=\pm 1$, then $(1,[\mathscr{E}])$ is a basis of $K_0(B)$ and
\begin{equation}
r^2-2r[\mathscr{E}]+[\mathscr{E}\otimes_{\mathscr{B}}\mathscr{E}]=0 ,
\end{equation}
where $r:=\mathrm{ch}_0(\mathscr{E})$.
\end{prop}
\endgroup

Applications of the last proposition include the standard Podle\'s sphere $S^2_q$, with $\mathscr{E}$ the bimodule associated to the $\mathcal{O}(U(1))$-Galois extension $\mathcal{O}(S^2_q)\subset\mathcal{O}(SU_q(2))$ and to the fundamental corepresentation of $\mathcal{O}(U(1))$, and the quantum sphere $S^4_q$ in \cite{LPR06}, with $\mathscr{E}$ the bimodule associated to the $\mathcal{O}(SU_q(2))$-Galois extension $\mathcal{O}(S^4_q)\subset\mathcal{O}(S^7_q)$ of \cite{LPR06} and to the fundamental corepresentation of $\mathcal{O}(SU_q(2))$. The latter example is described in details in the last part of the paper.

\noindent
{\bf Notations.}
If $a$ is an $n\times m$ matrix with entries in a ring $R$, we denote by $a_i^j$ the element
in the row $i$ and column $j$. So, the matrix product of two such matrices $a,b$ (with the correct number of rows/columns) has elements $(ab)_i^j=\sum_{k}a_i^kb_k^j$. 
We also define the dotted tensor product as $(a\,\dot{\ot}\,b)_i^j=\sum_ka_i^k\ot b_k^j$.
If $M$ is an $R$-module, we denote by $M^\vee:=\mathrm{Hom}_{\Z}(M,\Z)$ its dual.
If $\mathcal{H}$ is a Hilbert space, $\tr_{\mathcal{H}}:\mathcal{L}^1(\mathcal{H})\to\C$ denotes the trace on $\mathcal{H}$. 
We denote by $\tr:M_n(R)\to R$ the matrix trace, given by $\tr(a):=\sum_{k=1}^na_k^k$ for any square matrix $a\in M_n(R)$.
Finally, $\N$ is the set of natural numbers including $0$.

\section{The K-ring of a rational homology sphere}\label{sec:2}

If $X$ is a compact Haudorff space, the Chern character defines a homomorphism of graded rings from K-theory to singular cohomology:
\[
\mathrm{ch}:K_*(C(X))=K^*(X)\to H^*(X;\mathbb{Q}) \; .
\]
The ring structure in K-theory is induced by the direct sum and tensor product of vector bundles. Explicitly,
\begin{align}
\mathrm{ch}([E_1\oplus E_2]) &=\mathrm{ch}([E_1])+\mathrm{ch}([E_2]) \; , \label{eq:chernA} \\
\mathrm{ch}([E_1\otimes E_2]) &=\mathrm{ch}([E_1])\mathrm{ch}([E_2]) \; , \label{eq:chernB}
\end{align}
for  complex vector bundles $E_1$ and $E_2$ over $X$. An alternative point of view, if $X$ is a $2n$-dimensional smooth manifold (here we are only interested in the even-dimensional case), is to think of $\mathrm{ch}$ as mapping to the de Rham cohomology. Notice that $\mathrm{ch}=\mathrm{ch}_0+\mathrm{ch}_1+\ldots+\mathrm{ch}_n$
where $\mathrm{ch}_k([E])$ is a differential form of degree $2k$. We shall pass with ease from one picture to the other.

Assume that $X$ is connected and oriented (in addition to being compact).
Under the usual identification of $H^0(X;\mathbb{Q})$ with $\mathbb{Q}$, $\mathrm{ch}_0([E])$ becomes an integer and is equal to the \emph{rank} of the bundle $E$.
The top Chern character can be paired with the orientation homology class, or integrated over the manifold,
to give another integer called the (top) \emph{Chern number} of the vector bundle.
In gauge theories, if $n=1$ the Chern number computes the \emph{monopole charge} and if $n=2$ it computes the \emph{instanton number} of the vector bundle.

We wish to characterize those manifolds whose $K^0$-ring is the ring of dual numbers.

Recall that a $2n$-dimensional manifold $X$ is called a \emph{rational homology sphere} if it has the same rational homology groups of $S^{2n}$. A rational homology sphere that is not homeomorphic to a sphere is, for example, the Grassmannian $Gr_2(\R^4)$.
The study of solutions of Yang-Mills equations on rational homology spheres is an open line of research (see for instance \cite{DE22}).

\begin{prop}\label{prop:before}
Let $X$ be a $2n$-dimensional oriented closed manifold, $n\geq 1$. Then, the following are equivalent:
\begin{enumerate}
\item\label{en:one} $K^*(X)\cong\Z^2$ as an abelian group,

\item\label{en:two} $K^*(X)\cong\Z[t]/(t^2)$ as a unital ring,

\item\label{en:three} $X$ is a rational homology sphere and $K^*(X)$ has no torsion.
\end{enumerate}
Moreover, any of the above conditions imply that $X$ is connected and that $K^1(X)=0$.
\end{prop}

\begin{proof}
We use the fact that the Chern character extends to an isomorphism of graded rings
\begin{equation}\label{eq:Cherniso}
K^*(X)\otimes_{\Z}\mathbb{Q}\to H^*(X,\mathbb{Q}) .
\end{equation}
We shall also use the fact that, by Poincar\'e duality, $X$ is a rational homology sphere if and only if $H^k(X;\mathbb{Q})=0$ for all $0<k<2n$, and
$H^0(X;\mathbb{Q})\cong H^{2n}(X;\mathbb{Q})\cong\mathbb{Q}$. Note also that, for a rational homology sphere,
$H^*(X;\mathbb{Q})$ can be a graded ring in a unique way, that is: $H^*(X;\mathbb{Q})\cong\mathbb{Q}[t]/(t^2)$ where $t$ is any non-zero element of degree $2n$.

\noindent
\ref{en:two}$\Rightarrow$\ref{en:one} is obvious.

\noindent
\ref{en:one}$\Rightarrow$\ref{en:three}. If $K^*(X)\cong\Z^2$ as an abelian group, we deduce from the Chern isomorphism \eqref{eq:Cherniso} that the $\mathbb{Q}$-vector space $H^*(X;\mathbb{Q})$ has dimension $2$. 
Since $H_0(X;\mathbb{Q})$ and $H^0(X;\mathbb{Q})$ are not zero, by Poincar\'e duality $H^{2n}(X;\mathbb{Q})$ is also not zero.
Thus, it must be $H^0(X;\mathbb{Q})\cong H^{2n}(X;\mathbb{Q})\cong\mathbb{Q}$ (which in particular means that $X$ is connected), and $H^k(X;\mathbb{Q})=0$ for every other value of $k$.


\noindent
\ref{en:three}$\Rightarrow$\ref{en:two}. If $X$ is a rational homology sphere, $K^*(X)\otimes_{\Z}\mathbb{Q}\cong H^*(X;\mathbb{Q})\cong\mathbb{Q}[t]/(t^2)$. More precisely, $K^0(X)\otimes_{\Z}\mathbb{Q}\cong H^*(X;\mathbb{Q})\cong\mathbb{Q}[t]/(t^2)$ and $K^1(X)\otimes_{\Z}\mathbb{Q}=0$, since $t$ has even degree.
If $K^*(X)$ has no torsion, we deduce that $K^1(X)=0$ and $K^0(X)\cong\Z^2$ as an abelian group.

Any element in $K^0(X)$ can be expressed as a difference $[E]-[X\times\C^r]$ between the class of a vector bundle and the one of a trivial bundle. We shall write $[X\times\C^r]=r[X\times\C]=r$, denoting by $1$ the class of the trivial line bundle (the unit of the K-ring). Since \eqref{eq:Cherniso} is an isomorphism, we can find a vector bundle $E$ and an integer $r$ such that
$q\,\mathrm{ch}([E]-r)=t$ for some non-zero $q\in\mathbb{Q}$. Up to a rescaling, we can choose $t$ such that $q=-1$, and get
\[
\mathrm{ch}(r-[E])=t .
\]
Note that $\mathrm{ch}_0(r-[E])=0$ forces $r$ to be the rank of $E$, and the remaining condition is
$\mathrm{ch}_n([E])=-t$. The restriction of the Chern isomorphism to $K^0(X)$ gives an injective homomorphism of rings
\begin{equation}\label{eq:issurj}
K^0(X)\to \mathbb{Q}[t]/(t^2) ,\qquad\quad
j+k[E] \mapsto (j+rk)-kt ,
\end{equation}
where $j,k\in\Z$. Since the matrix
\[
\mat{r}{ 1 & r \\ 0 & -1 }
\]
is invertible over the integers, the homomorphism \eqref{eq:issurj} has image $\Z[t]/(t^2)$.
\end{proof}

The next proposition tells us how to construct the generator $t$ of Proposition \ref{prop:before}.

\begin{prop}\label{prop:next}
Let $X$ be a $2n$-dimensional oriented closed manifold, $n\geq 1$. 
Assume that $K^0(X)\cong\Z^2$ as an abelian group, and that there is a complex vector bundle $E\to X$ with rank $r$ and Chern number $c=\pm 1$. Then,
$K^0(X)=\Z[t]/(t^2)$ with $t=r-[E]$.
\end{prop}

\begin{proof}
We are under the hypothesis of Proposition \ref{prop:before}. We already know that there is an isomorphism of rings $K^0(X)\cong\Z[t]/(t^2)$. Let us identify $K^0(X)$ with $\Z[t]/(t^2)$.

The rank and Chern number define elements $\varphi^1,\varphi^2\in K^0(X)^\vee$.
Let $E$ be the bundle in the statement of the proposition. The matrix
\[
\mat{c}{
\varphi^1(1) & \varphi^2(1) \\
\varphi^1([E]) & \varphi^2([E]) \\
}=\mat{c}{
1 & 0 \\
r & c
}
\]
is invertible over the integer (it has determinant $c=\pm 1$). It follows from Lemma \ref{lemma:1} that the pair $(1,[E])$ is a basis of the $\Z$-module $K^0(X)$, and the pair $(\varphi^1,\varphi^2)$ is a basis of $K^0(X)^\vee$.
One has
\[
t=j+s[E]
\]
for some $j,s\in\Z$. Since the pair $(1,t)$ is also a basis, the matrix
\[
\mat{c}{ 1 & 0 \\ j & s }
\]
transforming one basis into the other must be invertible over the integers, that means that $s=\pm 1$ is a sign.
Note that
\[
j^2 +2js [E] + [E\otimes E]=t^2=0.
\]
Since $E\otimes E$ has rank $r^2$, applying $\varphi^1$ to both sides we find
\[
j^2 +2js r + r^2=(j+sr)^2=0 \;,
\]
that means $j=-sr$ and $t=-s(r-[E])$. Since $t\mapsto -t$ defines an automorphism of $\Z[t]/(t^2)$, up to a reparametrization we can choose $t=r-[E]$.
\end{proof}

\begin{ex}\label{ex:S2}
Let $X=S^2$ and let $E$ be the vector bundle associated to the Hopf fibration \[S^3\xrightarrow{U(1)}S^2\] and to the
defining representation of $U(1)$. One checks that the hypothesis of Proposition \ref{prop:next} are satisfied, and find the well-known isomorphism $K^0(S^2)\cong\Z[t]/(t^2)$ with $t=1-[E]$ (see for instance \cite[Cor.~2.2.2]{Ati19}).
\end{ex}

\begin{ex}\label{ex:S4}
Let $X=S^4$ and let $E$ be the vector bundle associated to the Hopf fibration \[S^7\xrightarrow{SU(2)}S^4\] and to the
defining representation of $SU(2)$. One checks that the hypothesis of Proposition \ref{prop:next} are satisfied, and find the isomorphism $K^0(S^4)\cong\Z[t]/(t^2)$ with $t=2-[E]$.
\end{ex}

The Chern classes of the bundles in the previous two examples are well-known. The computation can be found in textbooks. In the following, we give the explicit computation for every even-dimensional sphere.


Recall that, if $E\to X$ is a vector bundle of rank $r$ and $\lambda^i(E)$ denotes its $i$-th exterior power, then the \emph{Euler class} of $E$ in $K^0(X)$ is given by (see e.g.~\cite[pag.~187]{Kar78}):
\[
\chi(E)=\sum_{i=0}^r(-1)^i[\lambda^i(E)] .
\]
In Example \ref{ex:S2}, we see that $t=1-[E]$ is the Euler class of the basic monopole bundle.
In Example \ref{ex:S4}, since line bundles on $S^4$ are trivial, for $\lambda^2(E)=E\wedge E$, one has $[\lambda^2(E)]=1$ and
$t=2-[E]=1-[E]+[\lambda^2(E)]$ is the Euler class of the instanton bundle.

We now discuss the example of the unit sphere $S^{2n}$ in $\R^{2n+1}$, for $n$ any positive integer.
With an abuse of notations, we denote by $x_1,\ldots,x_{2n+1}$ both the Cartesian coordinates on $\R^{2n+1}$ and their restriction to $S^{2n}$. We shall also need the standard generators $\gamma_1,\ldots,\gamma_{2n+1}\in M_{2^n}(\C)$ of the Clifford algebra of $\R^{2n+1}$. In terms of Pauli matrices
\[
\sigma_1:=\mat{c}{0 & 1 \\ 1 & 0} , \qquad\quad
\sigma_2:=\mat{c}{0 & -\sqrt{-1} \\ \sqrt{-1} & 0} , \qquad\quad
\sigma_3:=\mat{r}{1 & 0 \\ 0 & -1} ,
\]
and with the identification of $M_{2^n}(\C)$ with $\underbrace{M_2(\C)\otimes\ldots\otimes M_2(\C)}_{n\text{ times}}$,
one has
\[
\gamma_{2i+j}:=(\sigma_3)^{\otimes i}\otimes\sigma_j\otimes 1^{\otimes n-i-1}
\]
for all $0\leq i\leq n-1$ and $j\in\{1,2\}$, and
\[
\gamma_{2n+1}:=(\sigma_3)^{\otimes n} .
\]
Every vector bundle $E\to X$ can be dually described by its $C(X)$-module of continuous sections, which in turn is described by a projection $p\in M_N(C(S^{2n}))$ for some $N\geq 1$. With this in mind, we have the following.

\begin{prop}\label{prop:sfereclassiche}
The projection
\begin{equation}\label{eq:projs2n}
p:=\frac{1}{2}\Big(1+\sum_{i=1}^{2n+1}x_i\gamma_i\Big) \in M_N(C(S^{2n}))
\end{equation}
describes a rank $2^{n-1}$ vector bundle $E\to S^{2n}$ whose (top) Chern number is $(-1)^n$. As a corollary,
one has the isomorphism of rings $K^0(S^{2n})\cong\Z[t]/(t^2)$ with $t=2^{n-1}-[E]$.
\end{prop}

The interest reader can find the proof of the previous proposition (the computation of the Chern number) in Appendix~\ref{AppB}. Note that the projection \eqref{eq:projs2n} is not only continuous, but linear in the Cartesian coordinates on the ambient Euclidean space.

Again the isomorphism $K^0(S^{2n})\cong\Z[t]/(t^2)$ is well-known, but Proposition \ref{prop:next} allows one to prove it in a completely algebraic way (by a simple index computation), and to show that \mbox{$t=2^{n-1}-[E]$}, with $E$ the bundle associated to the projection \eqref{eq:projs2n}.
This should be compared with \cite{SBS23}, where it is shown
 that the same bundle $E$ is the non-trivial generator of $K^0(S^{2n})$ as an abelian group (but with a much more complicated proof, and with no mention to the ring structure).

\begin{rem}
The index computation in Proposition \ref{prop:sfereclassiche} is consistent with the table given in \cite[page 20]{HL04} for even-dimensional orthogonal quantum spheres. However, the computation in \cite{HL04} is only valid if the deformation parameter $q\neq 1$, so the classical result cannot be derived from the quantum case.
\end{rem}

\section{Characteristic classes of noncommutative vector bundles}\label{sec:3}
In this section we work in the following framework.
We have a unital *-algebra $\mathscr{B}$, and its C*-enveloping algebra $B$ (we assume that it exists).
One should think of $B$ as describing a compact quantum space, and $\mathscr{B}$ as describing some additional structure on this quantum space. If $B=C(X)$ is commutative, and $X$ is a smooth manifold, $\mathscr{B}$ could be the class of smooth functions; if $X$ is an algebraic variety, $\mathscr{B}$ could be the class of polynomial functions; etc.
In addition, we have an inclusion of unital *-algebras
\[
\mathscr{B}\subseteq\mathscr{A} ,
\]
where $\mathscr{A}$ describes some auxiliary noncommutative space used for computations.

We are interested in classes in $K_0(B)$ that are represented by finitely generated projective right $\mathscr{B}$-modules (``noncommutative vector bundles'').

In Sect.~\ref{sec:3.1} we recall the construction of $\Z$-linear maps $K_0(B)\to\Z$ induced by an even Fredholm module. We shall focus on 1-summable Fredholm modules.
In Sect.~\ref{sec:3.2} we discuss some properties of $\mathscr{B}$-bimodules that are finitely generated and projective as right $\mathscr{B}$-modules.
In Sect.~\ref{sec:loctriv}, we specialize to right modules that ``trivialize'' over $\mathscr{A}$ (that is, their ``pullback'' to the total space is trivial), and denote by $\mathrm{Vect}_{\mathscr{A}}(\mathscr{B})$ their class.
In Sect.~\ref{sec:3.3}, under suitable hypothesis, we construct an isomorphism
\[
\mathrm{ch}:K_0(B)\to\Z[t]/(t^2)
\]
of abelian groups that is ``multiplicative'', in the sense that it is compatible with the tensor product in $\mathrm{Vect}_{\mathscr{A}}(\mathscr{B})$ (Proposition \ref{prop:undertheassumption} and \ref{prop:2q}).

\subsection{1-summable even Fredholm modules}\label{sec:3.1}

Let $B$ be a unital C*-algebra, $\pi_1,\pi_2:B\to\mathcal{B}(\mathcal{H})$ be two bounded *-representations on a Hilbert space $\mathcal{H}$ and
$F,\gamma\in\mathcal{B}(\mathcal{H}\otimes\C^2)$ the operators
\[
F=\mat{c}{0 & 1 \\ 1 & 0} \;,\qquad\quad
\gamma=\mat{c}{1 & 0 \\ 0 & -1} \;,
\]
where $1$ is the identity on $\mathcal{H}$. Let $\pi$ be the representation of $B$ on $\mathcal{H}\otimes\C^2$ which is the direct sum of $\pi_1$ and $\pi_2$.
Note that
\[
\gamma F[F,\pi(b)]=\mat{c}{\pi_1(b)-\pi_2(b) & 0 \\ 0 & \pi_1(b)-\pi_2(b)} .
\]
If the difference $\pi_1(b)-\pi_2(b)$ is of trace class on $\mathcal{H}$ for all $b$ in a dense unital *-subalgebra $\mathscr{B}$ of $B$, then $(B,\mathcal{H}\otimes\C^2,F,\gamma)$ is a $1$-summable even Fredholm module which
defines a $\Z$-module map $\varphi:K_0(B)\to\Z$. For $[p]$ the K-theory class of a projection $p\in M_N(\mathscr{B})$, this is given by
\[
\varphi([p])=\tfrac{1}{2}\,\tr_{\mathcal{H}\otimes\C^2}(\gamma F[F,\pi(\tr\,p)])=\tr_{\mathcal{H}}\big(
\pi_1(\tr\,p)-\pi_2(\tr\,p)
\big) \;,
\]
where $\tr\,p=\sum_{i=1}^Np_i^i\in\mathscr{B}$ and $\tr_{\mathcal{H}}$ (resp.~$\tr_{\mathcal{H}\otimes\C^2}$)
is the trace of operators on the Hilbert space $\mathcal{H}$ (resp.~$\mathcal{H}\otimes\C^2$).
We emphasize that $\varphi([p])$ only depends on the K-theory class of $p$, and it is an integer because it is the index of a Fredholm operator associated to the Fredholm module. Thus,
\[
\varphi\in K_0(B)^\vee .
\]
(In general, in this way one constructs a $\Z$-module map $K^0(B)\to K_0(B)^\vee$ which needs not to be injective or surjective. See the comment in Appendix \ref{sec:app}.)

\subsection{Tensor products}\label{sec:3.2}

In this section $\mathscr{B}$ is a unital algebra.
Recall that a right $\mathscr{B}$-module $\mathscr{E}$ is finitely generated and projective if and only if there exists finitely many elements $\xi^1,\ldots,\xi^N\in\mathscr{E}$
and finitely many right $\mathscr{B}$-linear maps $\psi_1,\ldots,\psi_N\in
\mathscr{E}^\vee=\mathrm{Hom}_{\mathscr{B}}(\mathscr{E},\mathscr{B})$ such that $\sum_{i=1}^N\xi^i\psi_i$ is the identity on $\mathscr{E}$.
With these data one constructs right $\mathscr{B}$-module maps
\begin{align}
\psi &:\mathscr{E}\longrightarrow\mathscr{B}^N \;, && e\longmapsto\begin{pmatrix}\psi_1(e) \\ \vdots \\ \psi_N(e) \end{pmatrix} \;, \label{eq:isoA} \\
\xi &:\mathscr{B}^N\longrightarrow\mathscr{E} \;, && \begin{pmatrix}b_1 \\ \vdots \\ b_N \end{pmatrix}
\longmapsto \sum_{j=1}^N\xi^jb_j \;, \label{eq:isoB}
\end{align}
such that $\xi\circ\psi=\id_{\mathscr{E}}$.
The map $\psi$ induces an isomorphism (of right $\mathscr{B}$-modules) between $\mathscr{E}$ and its image $p\,\mathscr{B}^N$ in $\mathscr{B}^N$, with the projection $p=(p_i^j)\in M_N(\mathscr{B})$ explicitly given by
\[
p_i^j=\psi_i(\xi^j) .
\]

\begin{lemma}\label{lemma:willrephrase}
Let $\mathscr{E}$ and $\mathscr{E}'$ be two $\mathscr{B}$-bimodules that are finitely generated and projective as right $\mathscr{B}$-modules. Then, $\mathscr{E}\otimes_{\mathscr{B}}\mathscr{E}'$ is finitely generated and projective as a right $\mathscr{B}$-module.
\end{lemma}

\begin{proof}
Choose elements $\xi^1,\ldots,\xi^N\in\mathscr{E}$ and $\psi_1,\ldots,\psi_N\in
\mathscr{E}^\vee$ and do the same for $\mathscr{E}'$ (with obvious notations).
Let
\[
\xi^{jk}:=\xi^j\otimes_{\mathscr{B}}\xi'^k \in\mathscr{E}\otimes_{\mathscr{B}}\mathscr{E}' ,
\]
and define $\psi_{j,k}\in (\mathscr{E}\otimes_{\mathscr{B}}\mathscr{E}')^\vee$ as follows
\[
\psi_{j,k}:\mathscr{E}\otimes_{\mathscr{B}}\mathscr{E}'\longrightarrow\mathscr{B} \;,
\qquad\quad
e\otimes_{\mathscr{B}}e'\longmapsto\psi_k'\big(\psi_j(e)e'\big) .
\]
Note that $\psi_{j,k}$ is well-defined, since $\psi_j$ is right $\mathscr{B}$-linear
(so $\psi_{j,k}(eb\otimes_{\mathscr{B}}e')=\psi_{j,k}(e\otimes_{\mathscr{B}}be')$ for
$b\in\mathscr{B}$).
For all $e,e'$, by definition

\begin{align*}
\sum_{j,k}\xi^{j,k}\psi_{j,k}(e &\otimes_{\mathscr{B}}e')
=\sum_{j,k} \xi^j\otimes_{\mathscr{B}}\xi'^k\psi_k'\big(\psi_j(e)e'\big) \hspace*{-1cm} \\
&=\sum_{j,k} \xi^j\otimes_{\mathscr{B}}\psi_j(e)e'
&& \text{(since $\sum\nolimits_k\xi'^k\psi_k'$ is the identity)}
 \\
&=\sum_{j,k} \xi^j\psi_j(e)\otimes_{\mathscr{B}}e' 
&& \text{(the tensor product is balanced over $\mathscr{B}$)}
\\
&=\sum_{j,k} e\otimes_{\mathscr{B}}e' . && \text{(since $\sum\nolimits_j\xi^k\psi_k$ is the identity)}
\end{align*}
So $\sum_{j,k}\xi^{j,k}\psi_{j,k}$ is the identity on $\mathscr{E}\otimes_{\mathscr{B}}\mathscr{E}'$, thus proving the lemma.
\end{proof}

In the notations of the above proof, the projection corresponding to the right $\mathscr{B}$-module $\mathscr{E}\otimes_{\mathscr{B}}\mathscr{E}'$ has matrix elements
\begin{equation}\label{eq:wesee}
\psi_{j_1,j_2}(\xi^{k_1,k_2})=
\psi_{j_2}'\big(\psi_{j_1}(\xi^{k_1})\xi'^{k_2}\big)
=\psi_{j_2}'\big(p_{j_1}^{k_1}\xi'^{k_2}\big) .
\end{equation}
where now rows and columns are labelled by pairs $(j_1,j_2)$ and $(k_1,k_2)$ in
$\{1,\ldots,N\}\times\{1,\ldots,N'\}$. We see that \eqref{eq:wesee} depends not only on the projections $p$ and $p'$ of $\mathscr{E}$ and $\mathscr{E}'$, but also on the explicit choice of isomorphism between $\mathscr{E}'$ and $p'\,\mathscr{B}^{n'}$.

In particular, the matrix trace of the projection \eqref{eq:wesee} is given by
\begin{equation}\label{eq:matrixtrace}
\sum_{j_1,j_2}\psi_{j_1,j_2}(\xi^{j_1,j_2})=
\sum_{j_2}\psi_{j_2}'\big(\tr\, p\,\xi'^{j_2}\big) .
\end{equation}
We see that this depends on the trace $\tr\, p=\sum_{j_1}p_{j_1}^{j_1}\in\mathscr{B}$ of $p$, and on the choice of isomorphism between $\mathscr{E}'$ and $p'\,\mathscr{B}^{N'}$.

\subsection{A class of noncommutative vector bundles}\label{sec:loctriv}

Suppose $E\to X$ is a vector bundle over a compact manifold.
One can always construct a compact space $Y$ with a quotient map $Y\to X$ such that the pullback of $E$ to $Y$ is trivial. For example, by the shrinking lemma one can find a finite cover $\{K_i\}$ by compact sets where the bundle trivializes, and define $Y$ as the topological disjoint union of the $K_i$'s.

In a dual language, if $\mathscr{E}$ is the $C(X)$-module of continuous sections of the bundle $E$, the $C(Y)$-module of continuous sections of the pullback bundle is $\mathscr{E}\otimes_{C(X)}C(Y)$, and its global triviality means that it is of the form $C(Y)^r$ for some positive $r\in\Z_+$.

In the noncommutative case, we fix an inclusion of unital algebras
$\mathscr{B}\subseteq\mathscr{A}$,
and consider the following type of object $\mathscr{E}$:
\begin{enumerate}
\item\label{en:Eone} $\mathscr{E}$ is a $\mathscr{B}$-bimodule that is finitely generated and projective as a right $\mathscr{B}$-module,

\item\label{en:Etwo} $\mathscr{E}$ trivializes over $\mathscr{A}$. This means that there is an isomorphism $\mathscr{E}\otimes_{\mathscr{B}}\mathscr{A}\cong\mathscr{A}^r$ with a free right $\mathscr{A}$-module ($r\in\Z_+$).

\end{enumerate}
We can identify $e\in\mathscr{E}$ with $e\otimes_{\mathscr{B}}1$ and $\mathscr{E}$ itself with a vector subspace of $\mathscr{A}^r$.
We denote by $\mathrm{Vect}_{\mathscr{A}}(\mathscr{B})$ the collection of all isomorphism classes of bimodules $\mathscr{E}$ satisfying \ref{en:Eone} and \ref{en:Etwo}. To ease the notations,
we denote by the same symbol $\mathscr{E}$ a bimodule and its class in $\mathrm{Vect}_{\mathscr{A}}(\mathscr{B})$. Since isomorphism as bimodules is stronger than isomorphism as right modules, there is a well-defined map from $\mathrm{Vect}_{\mathscr{A}}(\mathscr{B})$ to the collection of isomorphism classes of finitely generated projective right $\mathscr{B}$-modules, and then to $K_0(B)$.

\begin{prop}\label{prop:uv}
For every $\mathscr{E}\in\mathrm{Vect}_{\mathscr{A}}(\mathscr{B})$ there exists an $N\times r$ matrix $u$ and an $r\times N$ matrix $v$ with entries in $\mathscr{A}$ such that:
\begin{enumerate}[label=\textup{(\arabic*)}]
\item $v \, u=1$ \ is the $r\times r$ identity matrix,

\item\label{en:point2} $p:=u \, v$ \ belongs to $M_N(\mathscr{B})$,

\item\label{en:point3} the map \eqref{eq:isoA} is given by the multiplication from the left by the matrix $u$, and the map in \eqref{eq:isoB} is given by the multiplication from the left by the matrix $v$,

\item the two maps at point \ref{en:point3} restrict to isomorphisms $\mathscr{E}\cong p\,\mathscr{B}^N$ with $p$ the idempotent at point \ref{en:point2}.

\end{enumerate}
\end{prop}

\begin{proof}
Choose $\xi^1,\ldots,\xi^N\in\mathscr{E}$
and $\psi_1,\ldots,\psi_N\in\mathscr{E}^\vee$ such that $\sum_{i=1}^N\xi^i\psi_i=\id_{\mathscr{E}}$.
Under the identification of $\mathscr{E}$ with a vector subspace of $\mathscr{A}^r$,
the elements $\xi^1,\ldots,\xi^N$ are the columns of an $r\times N$ matrix $v$ with entries in $\mathscr{A}$.

The maps \eqref{eq:isoB} extend to an isomorphism of right $\mathscr{A}$-modules:
\[
p\,\mathscr{A}^N=p\,\mathscr{B}^N\otimes_{\mathscr{B}}\mathscr{A}\longrightarrow\mathscr{E}\otimes_{\mathscr{B}}\mathscr{A}=\mathscr{A}^r \;, \qquad\quad \begin{pmatrix}a_1 \\ \vdots \\ a_N \end{pmatrix}
\longmapsto v\cdot\begin{pmatrix}a_1 \\ \vdots \\ a_N \end{pmatrix} ,
\]
where $p=(p_i^j)\in M_N(\mathscr{B})\subseteq M_N(\mathscr{A})$ is given by $p_i^j=\psi_i(\xi^j)$ as before.
Since the previous map is an isomorphism, there is a (unique) set of elements $\{u^1,\ldots,u^r\}$ in $p\,\mathscr{A}^N$ whose image
is the canonical basis of $\mathscr{A}^r$. If we call $u$ the matrix with columns $u^1,\ldots,u^r$, then $u$ is (uniquely)
defined by the condition that
$v\, u=1$
plus the condition that each column belongs to $p\,\mathscr{A}^N$. The inverse isomorphism
\[
\mathscr{A}^r\longrightarrow p\,\mathscr{A}^N \;, \qquad\qquad \begin{pmatrix}a_1 \\ \vdots \\ a_r \end{pmatrix}
\longmapsto u\begin{pmatrix}a_1 \\ \vdots \\ a_r \end{pmatrix}
\]
is given by the left multiplication by $u$.

Since $v$ restricts to the isomorphism \eqref{eq:isoB} from $p\,\mathscr{B}^N$ to $\mathscr{E}$, it follows that
$u$ restricts to the inverse isomorphism \eqref{eq:isoA} from $\mathscr{E}$ to $p\,\mathscr{B}^N$. Thus,
\[
\begin{pmatrix}\psi_1(e) \\ \vdots \\ \psi_N(e) \end{pmatrix}=u e
\]
for all $e\in\mathscr{E}$. In particular, $p_i^j=\psi_i(\xi^j)=\sum_{k=1}^r u_i^k v_k^j$, that is $p=u\, v$.
\end{proof}

To simplify the discussion, for lack of a better name, we call the pair $(u,v)$ in Proposition \ref{prop:uv} a \emph{trivializing pair} for $\mathscr{E}$ (or more precisely for the right $\mathscr{A}$-module $\mathscr{E}\otimes_{\mathscr{B}}\mathscr{A}$).

\begin{prop}\label{prop:UV}
If $\mathscr{E},\widetilde{\mathscr{E}}\in\mathrm{Vect}_{\mathscr{A}}(\mathscr{B})$, then 
$\mathscr{E}\otimes_{\mathscr{B}}\widetilde{\mathscr{E}}\in\mathrm{Vect}_{\mathscr{A}}(\mathscr{B})$. Moreover,
if $(u,v)$ is a trivializing pair for $\mathscr{E}$ and
$(\widetilde{u},\widetilde{v})$ is a trivializing pair for $\widetilde{\mathscr{E}}$, then a trivializing pair
$(U,V)$ for $\mathscr{E}\otimes_{\mathscr{B}}\widetilde{\mathscr{E}}$ is given by
\begin{equation}\label{eq:UV}
U_{j_1,j_2}^{k_1,k_2}=\widetilde{u}^{k_2}_{j_2}u^{k_1}_{j_1} , \qquad\quad
V_{j_1,j_2}^{k_1,k_2}=v^{k_1}_{j_1}\widetilde{v}^{k_2}_{j_2} ,
\end{equation}
where now rows and columns are labelled by pairs $(j_1,j_2)$ and $(k_1,k_2)$.
\end{prop}

\begin{proof}
The composition of isomorphisms (of right $\mathscr{A}$-modules)
\[
(\mathscr{E}\otimes_{\mathscr{B}}\widetilde{\mathscr{E}})\otimes_{\mathscr{B}}\mathscr{A}
\cong
(\mathscr{E}\otimes_{\mathscr{B}}\mathscr{A})\otimes_{\mathscr{A}}(\widetilde{\mathscr{E}}\otimes_{\mathscr{B}}\mathscr{A})
\longrightarrow\mathscr{A}^r\otimes_{\mathscr{A}}\mathscr{A}^{\widetilde{r}}\cong \mathscr{A}^{r\widetilde{r}}
\]
is obviously an isomorphism, proving that $\mathscr{E}\otimes_{\mathscr{B}}\widetilde{\mathscr{E}}\in\mathrm{Vect}_{\mathscr{A}}(\mathscr{B})$. The rest of the proof is a simple rephrasing of the proof of Lemma \ref{lemma:willrephrase}. For example the maps
\[
\psi_{j_1,j_2}:\mathscr{E}\otimes_{\mathscr{B}}\widetilde{\mathscr{E}}\longrightarrow\mathscr{B}
\]
are now given by the restriction of the maps
\begin{align*}
(\mathscr{E}\otimes_{\mathscr{B}}\widetilde{\mathscr{E}})\otimes_{\mathscr{B}}\mathscr{A}
&\cong
(\mathscr{E}\otimes_{\mathscr{B}}\mathscr{A})\otimes_{\mathscr{A}}(\widetilde{\mathscr{E}}\otimes_{\mathscr{B}}\mathscr{A})
\longrightarrow\mathscr{A} ,
\\
e\otimes_{\mathscr{A}}e' &\longmapsto \sum_{i_1=1}^{r}\sum_{i_2=1}^{\widetilde{r}}\widetilde{u}^{i_2}_{j_2}u^{i_1}_{j_1}e_{i_1}e'_{i_2} ,
\end{align*}
where $e=(e_i)\in\mathscr{A}^r$ and $e'=(e'_i)\in\mathscr{A}^{\widetilde{r}}$. This proves the formula \eqref{eq:UV} for $U$.
Columns of the matrix $V$ in \eqref{eq:UV}, on the other hand, are mapped to the elements of the canonical basis of $\mathscr{A}^{r\widetilde{r}}$. This is a consequence of the identity $VU=1$ (which can be easily checked), and proves that the formula \eqref{eq:UV} for $V$ is also correct.
\end{proof}

Formulas \eqref{eq:wesee} and \eqref{eq:matrixtrace} now take the following form.

\begin{prop}
Let $\mathscr{E},\widetilde{\mathscr{E}}\in\mathrm{Vect}_{\mathscr{A}}(\mathscr{B})$, 
$\mathscr{E}\cong p\,\mathscr{B}^N$ as a right $\mathscr{B}$-module, and let
$(\widetilde{u},\widetilde{v})$ be a trivializing pair for $\widetilde{\mathscr{E}}$.
Then, $\mathscr{E}\otimes_{\mathscr{B}}\widetilde{\mathscr{E}}\cong P\,\mathscr{B}^{N\widetilde{N}}$
where $P=(p_{j_1,j_2}^{k_1,k_2})$ has matrix entries
\begin{equation}\label{eq:capitalP}
P_{j_1,j_2}^{k_1,k_2}=\sum_{i=1}^{\widetilde{r}}\widetilde{u}^{i}_{j_2}p_{j_1}^{k_1}\widetilde{v}^{k_2}_{i} .
\end{equation}
In particular,
\begin{equation}\label{eq:capitalTrace}
\tr\,P=\sum_{i=1}^{\widetilde{r}}\sum_{j=1}^{\widetilde{N}}\widetilde{u}^{i}_{j}(\tr\,p)\widetilde{v}^{j}_{i} .
\end{equation}
\end{prop}

\begin{proof}
This is a consequence of \eqref{eq:UV} and of $P=UV$.
\end{proof}

\subsection{On a multiplicative structure of K-theory}\label{sec:3.3}
In this section we assume that we have an inclusion $\mathscr{B}\subseteq\mathscr{A}$ of unital *-algebras, and that the C*-enveloping algebra $B$ of $\mathscr{B}$ exists.
There is an obvious map
\begin{equation}\label{eq:themap}
\mathrm{Vect}_{\mathscr{A}}(\mathscr{B})\longrightarrow K_0(B) .
\end{equation}
If $\mathscr{E}\in\mathrm{Vect}_{\mathscr{A}}(\mathscr{B})$, there is $p\in M_N(\mathscr{B})$ such that $\mathscr{E}\cong p\,\mathscr{B}^N$ as right $\mathscr{B}$-modules. With the *-homomorphism $\mathscr{B}\to B$ we transform $p$ into an idempotent in $M_N(B)$, and denote by $[\mathscr{E}]\in K_0(B)$ the K-theory class of this idempotent.

The map \eqref{eq:themap} is a morphism of abelian groups (the operation being the direct sum), and sends the trivial bimodule $\mathscr{B}$ into the class $1\in K_0(B)$.

\begin{prop}\label{prop:undertheassumption}
Assume that there is a character $\varepsilon:\mathscr{A}\to\C$ and
a bounded unital *-representation $\pi:\mathscr{A}\to\mathcal{B}(\mathcal{H})$ such that
\[
\pi(b)-\varepsilon(b)\in\mathcal{L}^1(\mathcal{H}) \;,\quad\forall\;b\in\mathscr{B}.
\]
Then:
\begin{enumerate}
\item\label{en:provaA} There is a homomorphism of unital semirings
\[
\mathrm{ch}:\big(\mathrm{Vect}_{\mathscr{A}}(\mathscr{B}),\oplus,\otimes\big)  \to\Z[t]/(t^2) ,
\quad\quad
\mathrm{ch}(\mathscr{E})=\mathrm{ch}_0(\mathscr{E})+\mathrm{ch}_1(\mathscr{E})t ,
\]
given, for any $\mathscr{E}\cong p\,\mathscr{B}^N$, by the formulas
\begin{equation} \label{ch01}
\mathrm{ch}_0(\mathscr{E}):=\varepsilon(\tr\,p) , \qquad\quad
\mathrm{ch}_1(\mathscr{E})=\tr_{\mathcal{H}}\big(\pi(\tr\,p)-\varepsilon(\tr\,p)\big) .
\end{equation}

\item\label{en:provaB}
The underlying morphism of abelian semigroups is the composition of \eqref{eq:themap} with a homomorphism $[\mathrm{ch}]:K_0(B)\to\Z[t]/(t^2)$.

\item\label{en:provaC}
If $(u,v)$ is a trivializing pair for $\mathscr{E}$, then the number of column of $u$ is $r=\mathrm{ch}_0(\mathscr{E})$ and thus depends only on the K-theory class of $\mathscr{E}$.
\end{enumerate}
\end{prop}

\begin{proof}
$[\mathrm{ch}_0]$ and $[\mathrm{ch}_1]$ are the index maps of the Fredholm modules associated with the pair or representations $(\varepsilon,0)$ on $\C$ and $(\pi,\varepsilon)$ on $\mathcal{H}$. 
Here $\varepsilon(b)$ acts on $\mathcal{H}$ as a scalar multiple of the identity.
Note that the representations $\varepsilon$ and $\pi$ extend from $\mathscr{B}$ to the C*-enveloping algebra $B$, due to the universal property of the latter.

We have to prove \ref{en:provaC}, and that the map $\mathrm{ch}$ is compatible with the tensor product of elements of $\mathrm{Vect}_{\mathscr{A}}(\mathscr{B})$, which in components means that
\begin{align}
\mathrm{ch}_0(\mathscr{E}\otimes_{\mathscr{B}}\widetilde{\mathscr{E}}) &=\mathrm{ch}_0(\mathscr{E})\mathrm{ch}_0(\widetilde{\mathscr{E}})
\; , \label{eq:provaA} \\
\mathrm{ch}_1(\mathscr{E}\otimes_{\mathscr{B}}\widetilde{\mathscr{E}}) &=
\mathrm{ch}_0(\mathscr{E})\mathrm{ch}_1(\widetilde{\mathscr{E}})+\mathrm{ch}_1(\mathscr{E})\mathrm{ch}_0(\widetilde{\mathscr{E}})
\; , \label{eq:provaB}
\end{align}
for all $\mathscr{E},\widetilde{\mathscr{E}}\in\mathrm{Vect}_{\mathscr{A}}(\mathscr{B})$.
To do this, let us choose trivializing pairs $(u,v)$ for $\mathscr{E}$ and $(\widetilde{u},\widetilde{v})$ of $\widetilde{\mathscr{E}}$. Let $p:=uv$ and remember that
$\mathscr{E}\otimes_{\mathscr{B}}\widetilde{\mathscr{E}}\cong P\,\mathscr{B}^{N\widetilde{N}}$ with idempotent $P=UV$ given by \eqref{eq:capitalP}.

Note that \eqref{eq:provaA} follows from \ref{en:provaC} and from the formula \eqref{eq:UV}, saying that $U$ has $r\widetilde{r}$ columns, where $r$ is the number of columns of $u$ and $\widetilde{r}$ the number of columns of $\widetilde{u}$.
It remains to prove \ref{en:provaC} and \eqref{eq:provaB}.

Firstly, we have
\begin{align}\label{rhor}
\mathrm{ch}_0(\mathscr{E})& =\varepsilon(\tr\,p)=
\sum_{i=1}^r\sum_{j=1}^n\varepsilon(u_i^jv^i_j)=
\sum_{i=1}^r\sum_{j=1}^n\varepsilon(u_i^j)\varepsilon(v^i_j)  \\
& =\sum_{i=1}^r\sum_{j=1}^n\varepsilon(v^i_j)\varepsilon(u_i^j)=
\sum_{i=1}^r\sum_{j=1}^n\varepsilon(v^i_ju_i^j)=
\varepsilon(\tr\,vu)=\varepsilon(\tr\,1_r)=r , \nonumber
\end{align}
where we use the fact that $\varepsilon$ maps to $\C$, so that the factors in the above products can be reordered.

Next, from \eqref{eq:capitalTrace} we get the identity:
\[
\pi(\tr\,P)-\varepsilon(\tr\,P) =
\sum_{i=1}^{\widetilde{r}}\sum_{j=1}^{\widetilde{N}}\pi(\widetilde{u}^{i}_{j})\big(\pi(\tr\,p)-r\big)\pi(\widetilde{v}^{j}_{i})+r\big(\pi(\tr\,\widetilde{p})-\widetilde{r}\big) \;.
\]
where $\widetilde{p}=\widetilde{u}\widetilde{v}$. Since $\pi(\tr\,p)-r$ and $\pi(\tr\,\widetilde{p})-\widetilde{r}$ are of trace class,
each summand in the first line is of trace class, and the term in the second line is of trace class as well.
If we apply $\tr_{\mathcal{H}}$ to both sides of the equality and use the cyclic property of the trace we find:
\begin{align*}
\mathrm{ch}_1(\mathscr{E}\otimes_{\mathscr{B}}\widetilde{\mathscr{E}})
&=\sum_{i=1}^{\widetilde{r}}\sum_{j=1}^{\widetilde{N}}\tr_{\mathcal{H}}\pi(\widetilde{u}^{i}_{j})\big(\pi(\tr\,p)-r\big)\pi(\widetilde{v}^{j}_{i})+r\,\tr_{\mathcal{H}}\big(\pi(\tr\,\widetilde{p})-\widetilde{r}\big) \\
&=\sum_{i=1}^{\widetilde{r}}\sum_{j=1}^{\widetilde{N}}\tr_{\mathcal{H}}\pi(\widetilde{v}^{j}_{i})\pi(\widetilde{u}^{i}_{j})\big(\pi(\tr\,p)-r\big)+r\,\mathrm{ch}_1(\widetilde{\mathscr{E}}) \\
&=\tr_{\mathcal{H}}\pi(\tr\,\widetilde{v}\widetilde{u})\big(\pi(\tr\,p)-r\big)+r\,\mathrm{ch}_1(\widetilde{\mathscr{E}}) \;.
\end{align*}
But $\pi(\tr\,\widetilde{v}\widetilde{u})=\pi(\tr\,1_{\widetilde{r}})=\widetilde{r}$ is a number, so
\begin{align*}
\mathrm{ch}_1(\mathscr{E}\otimes_{\mathscr{B}}\widetilde{\mathscr{E}})
&=\widetilde{r}\,\tr_{\mathcal{H}}\big(\pi(\tr\,p)-r\big)+r\,\mathrm{ch}_1(\widetilde{\mathscr{E}}) \\
&=\widetilde{r}\,\mathrm{ch}_1(\mathscr{E})+r\,\mathrm{ch}_1(\widetilde{\mathscr{E}}) \;,
\end{align*}
which is exactly \eqref{eq:provaB}.
\end{proof}

\begin{cor}\label{cor:usingcor}
Under the assumptions of Proposition \ref{prop:undertheassumption},
for every $\mathscr{E}\in\mathrm{Vect}_{\mathscr{A}}(\mathscr{B})$ and every positive integer $k$, for the module
\[
\mathscr{E}^{\otimes_{\mathscr{B}}k}:=\underbrace{
\mathscr{E}\otimes_{\mathscr{B}}\mathscr{E}\otimes_{\mathscr{B}}\ldots\otimes_{\mathscr{B}}\mathscr{E}}_{k\text{ times}}\vspace{-5pt}
\]
one has
\[
\mathrm{ch}_0(\mathscr{E}^{\otimes_{\mathscr{B}}k})=\mathrm{ch}_0(\mathscr{E})^k \;,\qquad
\mathrm{ch}_1(\mathscr{E}^{\otimes_{\mathscr{B}}k})=k\,\mathrm{ch}_0(\mathscr{E})^{k-1}\mathrm{ch}_1(\mathscr{E}) \;.
\]
\end{cor}

\begin{proof}
It is an immediate consequence of \eqref{eq:provaA} and \eqref{eq:provaB}.
\end{proof}

Under the above assumptions, we have the following analogue of Proposition \ref{prop:next}.

\begin{prop}\label{prop:2q}
Under the assumptions of Proposition \ref{prop:undertheassumption},
if $K_0(B)\cong\Z^2$ and there exists $\mathscr{E}\in\mathrm{Vect}_{\mathscr{A}}(\mathscr{B})$ with
$\mathrm{ch}_1(\mathscr{E})=\pm 1$, then $(1,[\mathscr{E}])$ is a basis of $K_0(B)$ and
\begin{equation}\label{eq:identity}
r^2-2r[\mathscr{E}]+[\mathscr{E}\otimes_{\mathscr{B}}\mathscr{E}]=0 ,
\end{equation}
where $r:=\mathrm{ch}_0(\mathscr{E})$.
\end{prop}

The identity \eqref{eq:identity} is the analogue of the classical statement $t^2=0$ for $t:=r-[E]$.

\begin{proof}
Let $\varphi^1,\varphi^2\in K_0(B)^\vee$ be the maps $\varphi^1=[\mathrm{ch}_0]$ and $\varphi^2=[\mathrm{ch}_1]$
and let $c:=\mathrm{ch}_1(\mathscr{E})$.
Since
\[
\mat{c}{
\varphi^1(1) & \varphi^2(1) \\
\varphi^1([\mathscr{E}]) & \varphi^2([\mathscr{E}]) \\
}=\mat{c}{
1 & 0 \\
r & c
}
\]
is invertible over the integers, it follows from Lemma \ref{lemma:1} that the pair $(1,[\mathscr{E}])$ is a basis of $K_0(B)$ and
the pair $(\varphi^1,\varphi^2)$ is a basis of $K_0(B)^\vee$.

As a byproduct, an element $y\in K_0(B)$ is zero if and only if $\varphi^1(y)=\varphi^2(y)=0$. For $y:=r^2-2r[\mathscr{E}]+[\mathscr{E}\otimes_{\mathscr{B}}\mathscr{E}]$ the left hand side of \eqref{eq:identity}, applying $\varphi^1$ and $\varphi^2$ to both sides and using Corollary \ref{cor:usingcor} we find
$\varphi^1(y)=r^2-2r^2+r^2=0$ and $\varphi^2(y) =0-2rc+2rc=0$.
\end{proof}

\section{Modules associated to corepresentations}

In the main example we are interested in, the inclusion $\mathscr{B}\subseteq\mathscr{A}$ is a Hopf-Galois extension. We collect in this section some results about this case.

Let $H$ be a Hopf algebra with bijective antipode, $\mathscr{A}$ a right $H$-comodule algebra with coaction $\delta_R$ (all right coactions will be denoted $\delta_R$), $\mathscr{B}:=\mathscr{A}^{\mathrm{co}H}$ the subalgebra of $\mathscr{A}$ of coinvariant elements,
and assume that $\mathscr{B}\subset\mathscr{A}$ is a faithfully flat Hopf-Galois extension. Then the cotensor product 
$\mathscr{A}\Box^{H}(\,\,\cdot\,\,)$ is an additive monoidal functor from the category of left $H$-comodules to that of $\mathscr{B}$-bimodules; 
this means in particular that one has $\mathscr{B}$-bimodule isomorphisms
\begin{align}
\mathscr{A}\;\Box^{H}\,(V_1\oplus V_2) &\cong (\mathscr{A}\;\Box^{H}\,V_1)\oplus (\mathscr{A}\;\Box^{H}\,V_2) , 
\label{eq:monoidal-1}\\
\mathscr{A}\;\Box^{H}\,(V_1\otimes V_2) &\cong (\mathscr{A}\;\Box^{H}\,V_1)\otimes_{\mathscr{B}} (\mathscr{A}\;\Box^{H}\,V_2) , \label{eq:monoidal-2}
\end{align}
for any two left $H$-comodules $V_1$ and $V_2$ \cite[Thm.~2.5.3(2)]{Sch04}.
If $V$ is finite-dimensional, $\mathscr{A}\Box^{H}V$ is finitely generated and projective as a right (or left) $\mathscr{B}$-module \cite[Cor.~2.5.5]{Sch04}.
We now give a little more details about this result and describe how to construct a projection for the module associated to a finite-dimensional corepresentation.

Recall that, if $\mathscr{V}$ is an object in a monoidal category $(\mathcal{C},\otimes,\mathbf{1})$, a \emph{left dual} $\mathscr{V}^*$ (if it exists) is an object equipped with two morphisms \cite[Sect.~9.3]{Maj95}
\[
\mathrm{ev}_{\mathscr{V}}:\mathscr{V}^*\otimes\mathscr{V}\to\mathbf{1} , \qquad\quad
\mathrm{coev}_{\mathscr{V}}:\mathbf{1}\to\mathscr{V}\otimes\mathscr{V}^* ,
\]
satisfying the triangle identities
\begin{center}
\begin{tikzpicture}[scale=2]

\node (a) at (0,0) {$\mathbf{1}\otimes\mathscr{V}$};
\node (b) at (1,1) {$\mathscr{V}\otimes\mathscr{V}^*\otimes \mathscr{V}$};
\node (c) at (2,0) {$\mathscr{V}\otimes\mathbf{1}$};

\draw[-To,font=\footnotesize] (a) edge node[left] {$\mathrm{coev}_{\mathscr{V}}\otimes\id_{\mathscr{V}}$} (b) (b) edge node[right] {$\id_{\mathscr{V}}\otimes\mathrm{ev}_{\mathscr{V}}$} (c) (a) edge node[below] {$\id_{\mathscr{V}}$} (c);

\end{tikzpicture}\hspace{1cm}%
\begin{tikzpicture}[scale=2]

\node (a) at (2,0) {$\mathbf{1}\otimes\mathscr{V}^*$};
\node (b) at (1,1) {$\mathscr{V}^*\otimes\mathscr{V}\otimes \mathscr{V}^*$};
\node (c) at (0,0) {$\mathscr{V}^*\otimes\mathbf{1}$};

\draw[-To,font=\footnotesize] (b) edge node[right] {$\mathrm{ev}_{\mathscr{V}}\otimes\id_{\mathscr{V}^*}$} (a) (c) edge node[left] {$\id_{\mathscr{V}^*}\otimes\mathrm{coev}_{\mathscr{V}}$} (b) (c) edge node[below] {$\id_{\mathscr{V}^*}$} (a);

\end{tikzpicture}
\end{center}
where $\mathbf{1}$ is the tensor unit and, for every $\mathscr{V}$, we identify $\mathscr{V}\otimes\mathbf{1}$ and $\mathbf{1}\otimes\mathscr{V}$ with $\mathscr{V}$.

\begin{ex}\label{ex:comod}
Let $(V,\delta_L)$ be an object in the category ${}^{H}\mathrm{Mod}$ of finite-dimensional left $\mathcal{H}$-comodules with morphisms given by left $H$-colinear maps.
Let $V^\vee=\mathrm{Hom}_{\C}(V,\C)$ be the vector space dual (note that this consists of maps $V\to\C$ that are $\C$-linear, so not necessarily morphisms in  ${}^{H}\mathrm{Mod}$).
Fix a basis $(e_i)_{i=1}^r$ of $V$, denote by $(e^i)_{i=1}^r$ the dual basis of $V^\vee$, and call $T=(t_i^j)\in M_r(H)$ the matrix determined by the basis and the coaction via the formula
\begin{equation}\label{eq:matrixt}
\delta_L(e_i)=\sum_{j=1}^rt_i^j\otimes e_j .
\end{equation}
The dual coaction $\delta_L: V^\vee\to H\otimes V^\vee$ (all left coactions will be denoted by $\delta_L$) is given in the dual basis by
\[
\delta_L(e^j)=\sum_{j=1}^r S(t_i^j)\otimes e^i .
\]
Let $\mathrm{ev}_V:V^\vee\otimes V\to\C$ be the evaluation map and $\mathrm{coev}_V:\C\to V\otimes V^\vee$ the linear map determined by
\[
\mathrm{coev}_V(1)=\sum_{i=1}^re_i\otimes e^i .
\]
One easily checks that these two maps are morphisms in ${}^{H}\mathrm{Mod}$ (in particular, $\mathrm{coev}_V(1)$ is $H$-coinvariant). The triangle identities above are satisfied by definition of dual basis. Thus, $V^\vee$ equipped with the maps $\mathrm{ev}_V$ and $\mathrm{coev}_V$ is a left dual of $V$ in ${}^{H}\mathrm{Mod}$.
\end{ex}

\begin{ex}\label{ex:projmod}
Consider the monoidal category ${}_{\mathscr{B}}\mathrm{Mod}_{\mathscr{B}}$ of $\mathscr{B}$-bimodules, with morphisms given by $\mathscr{B}$-bilinear maps, with bifunctor $\otimes_{\mathscr{B}}$, and tensor unit $\mathscr{B}$.
For $\mathscr{E}$ an object in this category, let $\mathscr{E}^\vee:=\mathrm{Hom}_{\mathscr{B}}(\mathscr{E},\mathscr{B})$ as before (note that its elements are right $\mathscr{B}$-linear maps, so not necessarily morphisms in ${}_{\mathscr{B}}\mathrm{Mod}_{\mathscr{B}}$). The bimodule structure on $\mathscr{E}^\vee$ is
\begin{equation}\label{eq:dualbimo}
(b\cdot f)(\xi) :=b f(\xi) , \qquad\quad
(f\cdot b)(\xi) :=f(b\xi) ,
\end{equation}
for all $b\in\mathscr{B}$, $f\in\mathscr{E}^\vee$ and $\xi\in\mathscr{E}$.
Here $\cdot$ denotes the left/right action of $\mathscr{B}$ on $\mathscr{E}^\vee$, while in the displayed equation on the right hand side we have the multiplication in $\mathscr{B}$ and the left action on $\mathscr{E}$ respectively.
Clearly $b\cdot f$ and $f\cdot b$ are right $\mathscr{B}$-linear maps, and the left and right actions
defined in this way commute.

Let $\mathrm{ev}_{\mathscr{E}}:\mathscr{E}^\vee\otimes_{\mathscr{B}}\mathscr{E}\to\C$ be the evaluation map. Using \eqref{eq:dualbimo} one checks that this map is well-defined and is a morphism in ${}_{\mathscr{B}}\mathrm{Mod}_{\mathscr{B}}$.
Assume that there exists a map $\mathrm{coev}_{\mathscr{E}}:\mathscr{B}\to \mathscr{E}\otimes_{\mathscr{B}}\mathscr{E}^\vee$ that makes $\mathscr{E}^\vee$ a left dual of $\mathscr{E}$ in the category ${}_{\mathscr{B}}\mathrm{Mod}_{\mathscr{B}}$. Then, there exists  $\xi^1,\ldots,\xi^N\in\mathscr{E}$ and $\psi_1,\ldots,\psi_N\in\mathscr{E}^\vee$ such that
\begin{equation}\label{eq:coevE}
\mathrm{coev}_{\mathscr{E}}(1)=\sum_{i=1}^N\xi^i\otimes_{\mathscr{B}}\psi_i .
\end{equation}
The triangle identities imply that $\sum_{i=1}^N\xi^i\psi_i$ is the identity on $\mathscr{E}$, which means that $\mathscr{E}$ is finitely generated and projective as a right $\mathscr{B}$-module.
(One could also prove the converse: if $\mathscr{E}$ is finitely-generated and projective as a right $\mathscr{B}$-module, then $\mathscr{E}^\vee$ is a left dual of $\mathscr{E}$ in the category ${}_{\mathscr{B}}\mathrm{Mod}_{\mathscr{B}}$.)
\end{ex}

The argument in \cite{Sch04} is that the functor $\mathscr{A}\Box^{H}(\,\,\cdot\,\,)$ transforms
$V$ into $\mathscr{E}:=\mathscr{A}\Box^{H}V$, $V^\vee$ into a bimodule that is isomorphic to $\mathscr{E}^\vee$, and after the identification of $\mathscr{A}\Box^{H}V^\vee$ with $\mathscr{E}^\vee$ it transforms $\mathrm{ev}_V$ into $\mathrm{ev}_{\mathscr{E}}$ and $\mathrm{coev}_V$ into a coevaluation map $\mathrm{coev}_{\mathscr{E}}$.
Since functors (and isomorphisms) transform commutative diagram into commutative diagram,
the triangle identities satisfied by $\mathrm{ev}_V$ and $\mathrm{coev}_V$ imply that $\mathrm{ev}_{\mathscr{E}}$ and $\mathrm{coev}_V$ satisfy the triangle identities as well.

We now show that associated bundles trivialize when pulled back to the total space of a noncommutative principal bundle.

\begin{lemma}\label{lemma:pulltriv}
Let $V$ be a finite-dimensional left $H$-comodule, $(e_i)_{i=1}^r$ a basis of $V$, $T=(t^j_i)$ the matrix in \eqref{eq:matrixt}, and $\mathscr{E}:=\mathscr{A}\Box^{H}V$. Then, the multiplication map
\begin{equation}\label{eq:pulltriv}
\mathscr{E}\otimes_{\mathscr{B}}\mathscr{A}\longrightarrow\mathscr{A}\otimes V , \qquad\quad
\Big(\sum_ {i=1}^ra^i\otimes e_i\Big)\otimes_{\mathscr{B}}\widetilde{a}\longmapsto
\sum_ {i=1}^ra^i\widetilde{a}\otimes e_i
\end{equation}
is an isomorphism of right $\mathscr{A}$-modules.
\end{lemma}

\begin{proof}
Recall that the cotensor product is a bifunctor
\[
\Box^{H}:\mathrm{Mod}^{H}_{\mathscr{A}}\times 
{}^{H}\mathrm{Mod}\to\mathrm{Mod}_{\mathscr{A}} .
\]
By definition of Hopf-Galois extension, the canonical map
\[
\mathrm{can}:\mathscr{A}\otimes_\mathscr{B}\mathscr{A}\to \mathscr{A}\otimes H , \qquad\quad a\otimes_\mathscr{B}\widetilde{a}\mapsto (a\otimes 1)\delta_R(\widetilde{a}) ,
\]
is bijective. Since $H$ has bijective antipode, the map
\[
\mathrm{can}':\mathscr{A}\otimes_\mathscr{B}\mathscr{A}\to \mathscr{A}\otimes H , \qquad\quad a\otimes_\mathscr{B}\widetilde{a}\mapsto \delta_R(a) (\widetilde{a}\otimes 1) ,
\]
is also bijective \cite[Sect.~1.2]{KT81}. The vector space $\mathscr{A}\otimes_\mathscr{B}\mathscr{A}$ is an $\mathscr{A}$-bimodule in the obvious way, and carries a right coaction of $H$ given by the application of $\delta_R$ on the first leg:
\begin{equation}\label{eq:coactiononABA}
(a\otimes_{\mathscr{B}}\widetilde{a})\longmapsto
(a_{(0)}\otimes_{\mathscr{B}}\widetilde{a})\otimes a_{(1)} .
\end{equation}
The vector space $\mathscr{A}\otimes H$ carries a right coaction of $H$ given by the coproduct on the second leg, a right action of $\mathscr{A}$ given by the right multiplication on the first leg, and a left action of $a\in\mathscr{A}$ given by the multiplication from the left by $\delta_R(a)$.
With respect to these actions and coactions, $\mathrm{can}'$ is $\mathscr{A}$-bilinear and $H$-colinear.

Applying the functor $(\,\,\cdot\,\,)\Box^{H}V$ on both domain and codomain of $\mathrm{can}'$ we get an isomorphism of right $\mathscr{A}$-modules:
\begin{equation}\label{eq:canp}
\mathrm{can}'\,\Box^{H}V:
\mathscr{E}\otimes_\mathscr{B}\mathscr{A}=
(\mathscr{A}\Box^{H}V)\otimes_\mathscr{B}\mathscr{A}\to \mathscr{A}\otimes (H\Box^{H}V) .
\end{equation}
In $H\Box^{H}V$ the right coaction of $H$ on itself is given by the coproduct $\Delta$.
The counit induces linear map
\begin{equation}\label{eq:counit}
H\Box^{H}V\to V , \qquad\quad
\sum_{j=1}^rh^j\otimes e_j\mapsto\sum_{j=1}^r\varepsilon(h^j)e_j \, .
\end{equation}
On the other hand, the coaction $\delta_L:V\to H\otimes V$ has image in the cotensor product and induces a map that is a right inverse of the map \eqref{eq:counit}. It is not difficult to show that it is also a left inverse.
A sum $\sum_{j=1}^rh^j\otimes e_j\in H\otimes V$ belongs to $H\Box^{H}V$ if and only if
\[
\Delta(h^j)=\sum_{i=1}^rh^i\otimes t_i^j .
\]
From $(\varepsilon\otimes\id)\Delta(h^j)=h^j$ we deduce that $\sum_{i=1}^r\varepsilon(h^i)t_i^j=h^j$.
Using this, if we apply $\delta_L$ to the image of \eqref{eq:counit} we find
\[
\delta_L\Big(\sum_{j=1}^r\varepsilon(h^j)e_j\Big)=
\sum_{i=1}^r\varepsilon(h^i)t_i^j\otimes e_j=\sum_{j=1}^rh^j\otimes e_j ,
\]
thus proving the claim.

If we apply \eqref{eq:counit} to the parenthesis in \eqref{eq:canp}, we find the desired right $\mathscr{A}$-module isomorphism
\[
\mathscr{E}\otimes_\mathscr{B}\mathscr{A}\to \mathscr{A}\otimes (H\Box^{H}V)\to
\mathscr{A}\otimes V\to\mathscr{A}\otimes V ,
\]
given explicitly by the formula 
\begin{multline*}
\Big(\sum_ {i=1}^ra^i\otimes e_i\Big)\otimes_{\mathscr{B}}\widetilde{a}\mapsto
\sum_ {i=1}^r\Big(\delta_R(a^i)(\widetilde{a}\otimes 1)\Big) \otimes e_i \\
\mapsto
\sum_ {i=1}^r\Big((\id\otimes\varepsilon)\delta_R(a^i)(\widetilde{a}\otimes 1)\Big) \otimes e_i
=\sum_ {i=1}^ra^i\widetilde{a}\otimes e_i .
\end{multline*}
\end{proof}

Let $T=(t^j_i)$ be the matrix in \eqref{eq:matrixt}.
Using the dual bases in Example \ref{ex:comod}, any element in $\mathscr{E}:=\mathscr{A}\Box^{H}V$ can be written in a unique way as a sum
\begin{equation}\label{eq:bolda}
\sum_{j=1}^r a^j\otimes e_j
\end{equation}
with $a^1,\ldots,a^r\in\mathscr{A}$ satisfying
\begin{equation}\label{eq:boldad}
\delta_R(a^j)=\sum_{i=1}^r a^i\otimes t^j_i ,
\end{equation}
and any element in $\mathscr{A}\Box^{H}V^\vee$ can be written in a unique way as a sum
\begin{equation}\label{eq:boldat}
\sum_{i=1}^r \widetilde{a}_i\otimes e^i
\end{equation}
with $\widetilde{a}_1,\ldots,\widetilde{a}_r\in\mathscr{A}$ satisfying
\begin{equation}\label{eq:boldatd}
\delta_R(\widetilde{a}_i)=\sum_{i=1}^r \widetilde{a}_j\otimes S(t^j_i) .
\end{equation}
There is an obvious morphism $\phi:\mathscr{A}\Box^{H}V^\vee\to\mathscr{E}^\vee$ of $\mathscr{B}$-bimodules (in fact, an isomorphism), that maps \eqref{eq:boldat} to a linear map $\mathscr{E}\to\mathscr{B}$ consisting in the evaluation of the second leg of \eqref{eq:boldat} on the second leg of \eqref{eq:bolda}.

To construct an explicit projection now we would need a formula for the map
\begin{align*}
(\mathscr{A}\Box^{H}V_1)\otimes_{\mathscr{B}}(\mathscr{A}\Box^{H}V_2) &\to
\mathscr{A}\Box^{H}(V_1\otimes V_2) , \\
\big(\sum\nolimits_ia^i\otimes v_i\big)\otimes_{\mathscr{B}}
\big(\sum\nolimits_j\widetilde{a}^j\otimes w_j\big) &\mapsto
\sum\nolimits_{i,j}a^i\widetilde{a}^j\otimes (v_i\otimes w_j) .
\end{align*}
We know that this is an isomorphism for any pair of comodules $V_1$ and $V_2$, but it is not obvious how to get an explicit formula for its inverse. With that, $\mathrm{coev}_{\mathscr{E}}$ could be obtained as the composition:
\begin{align*}
\mathscr{B} &\stackrel{\sim}{\longrightarrow}\mathscr{A}\Box^{H}\C\xrightarrow{\mathscr{A}\Box^{H}\mathrm{coev}_V}\mathscr{A}\Box^{H}(V\otimes V^\vee)
\stackrel{\sim}{\longrightarrow}(\mathscr{A}\Box^{H}V)\otimes_{\mathscr{B}}(\mathscr{A}\Box^{H}V^\vee) \\
& \hspace*{8cm}\xrightarrow{\id\otimes_{\mathscr{B}}\phi}\mathscr{E}\otimes_{\mathscr{B}}\mathscr{E}^\vee
\end{align*}
where the first arrow is the map $b \mapsto b\otimes 1$.
In particular, $\mathrm{coev}_{\mathscr{E}}(1)$ is the image through the last two arrows of the element $1\otimes\mathrm{coev}_V(1)=1\otimes\sum_{i=1}^re_i\otimes e^i$.

In our main example in Sect.~\ref{sec:4} we shall adopt a different strategy and use the following lemma.
From now on, we assume that $H$ is a Hopf *-algebra and $\mathscr{A}$ a right $H$-comodule *-algebra, which means that $\delta_R(a^*)=\delta_R(a)^{*\otimes *}$. Also, we assume that $V$ is a unitary corepresentation, that is $T=(t_i^j)\in M_r(H)$ is a unitary matrix. By unicity of the inverse matrix, this means that $(t_j^i)^*=S(t_i^j)$.

\begin{lemma}\label{lemma:44}
Assume we can find a $N\times r$ matrix $u=(u_i^j)$ with entries in $\mathscr{A}$ such that
\[
\delta_R(u_i^j)=\sum_{k=1}^ru_i^k\otimes t_k^j \qquad\text{and}\qquad u^*u=1_r .
\]
Then, $\mathscr{A}\Box^{H}V^\vee\cong p\,\mathscr{B}^N$ with $p=uu^*$,
and $(u,v:=u^*)$ is a trivializing pair for this module.
\end{lemma}

\begin{proof}
By definition $u_j^i:=(v_i^j)^*$.
The entries
$p_i^j := \sum_{k=1}^r u_i^kv_k^j$
of the matrix $p=(p_i^j)$ are $\delta_R$-coinvariant by construction, that is they belong to $\mathscr{B}$. Clearly $p$ is a projection. Because of \eqref{eq:boldatd}, the right $\mathscr{B}$-module map
\[
\mathscr{A}\Box^{H}V^\vee\longrightarrow
\mathscr{A}^N , \qquad\quad
\sum_{j=1}^r a_j\otimes e^j\longmapsto
\Big(\sum_{j=1}^r u^j_ia_j \Big)_{i=1}^N
\]
has image in $p\,\mathscr{B}^N$ (the sum on the right hand side is $\delta_R$-coinvariant). The inverse map is given by
\[
p\,\mathscr{B}^N\longrightarrow\mathscr{A}\Box^{H}V^\vee , \qquad\quad
(b_1,\ldots,b_N)\longmapsto 
\sum_{i=1}^r \sum_{j=1}^N v_i^jb_j\otimes e^i \,,
\]
proving that we have an isomorphism.
\end{proof}

The previous lemma holds even if the extension $\mathscr{B}:=\mathscr{A}^{\mathrm{co}H}\subset\mathscr{A}$ is not Hopf-Galois. Under the same assumptions, one proves that 
$\mathscr{B}^Np$ is isomorphic to $\mathscr{A}\Box^{H}V$ as a left $\mathscr{B}$-module.

\section{Vector bundles on a quantum 4-sphere}\label{sec:4}

Let \mbox{$0<q<1$}.
We call \emph{quantum symplectic $7$-sphere} the noncommutative space underlying the unital *-algebra $\mathcal{O}(S^7_q)$ generated by elements $\{x_i,x_i^*\}_{i=1,\ldots,4}$ with commutation relations
\begin{align*}
x_1 x_2&=q x_2 x_1, & 
x_1 x_3 &=q x_3 x_1, & x_1 x_4&=q^2 x_4 x_1 ,
\\
x_2 x_3&=q^2 x_3 x_2+ q^2 (q-q^{-1}) x_4 x_1 ,  &
x_2 x_4&=q x_4 x_2 ,
& x_3 x_4&=q x_4 x_3 ,
\end{align*}
together with
\begin{align*}
x_1^* x_1 &= x_1  x_1^* , &
x_1^* x_4 &=q^2 x_4 x_1^* , &
x_2^* x_4 &=q x_4 x_2^*+q^2(q-q^{-1}) x_3 x_1^* ,
\\
x_1^* x_2 &=q x_2 x_1^* , &
x_2^* x_2 &= x_2  x_2^* + (1-q^2) x_1 x_1^* , &
x_3^* x_3 &= x_3  x_3^* +(1-q^4) x_2  x_2^* + (1-q^2) x_1 x_1^* ,
\\
x_1^* x_3 &=q x_3 x_1^*   , &
x_2^* x_3&=q^2 x_3 x_2^* , &
x_3^* x_4 &=q x_4  x_3^*-q^4(q-q^{-1}) x_2 x_1^* ,
\end{align*}
and sphere relations:
\begin{gather*}
x_1 x_1^*+x_2 x_2^*+ x_3 x_3^*+ x_4 x_4^*=1 , \\
q^8 x_1^* x_1+q^6 x_2^* x_2+q^2 x_3^* x_3+x_4^* x_4=1 .
\end{gather*}
Additional relations are obtained from the ones above by applying the algebra involution $^*$ (for example, $x_4^*  x_1=q^2 x_1 x_4^*$).
This algebra was studied in ~\cite{LPR06} as a quantum homogeneous space of the quantum symplectic group $\mathcal{O}(Sp_q(2))$. 
(The generators used in~\cite{LPR06}, here denoted $x'_i$, are related
to those used here by the equations $x'_1=q^4x_1$, $x'_2=q^3x_2$, $x'_3=qx_3$ and $x'_4=x_4$.)


A character $\varepsilon:\mathcal{O}(S^7_q)\to\C$ is given on generators by
\[
\varepsilon(x_i)=0\;  \quad \mbox{for } i\neq 4,\qquad
\varepsilon(x_4)=1 \;.
\]
The character $\varepsilon$ is the restriction to $\mathcal{O}(S^7_q)$ of the counit 
of the quantum group $\mathcal{O}(Sp_q(2))$. 

The quotient of $\mathcal{O}(S^7_q)$ by the ideal generated by $x_1$ is the unital *-algebra $\mathcal{O}(S^5_q)$ of a quantum $5$-sphere (see \cite{DL21} for the general case). The irreducible bounded *-representation of $\mathcal{O}(S^5_q)$ in \cite{DL21} (the one for $\lambda=1$), when pulled back to $\mathcal{O}(S^7_q)$ gives a representation
$\pi$ on $\ell^2(\N^2)$, with canonical orthonormal basis $\big(\ket{k_1,k_2}\big)_{k_1,k_2\in\N}$,
 that on generators reads:
\begin{align*}
\pi(x_1) &=0  \\ 
\pi(x_2)\ket{k_1,k_2} &=q^{k_1+2k_2}\ket{k_1,k_2} \\
\pi(x_3)\ket{k_1,k_2} &=q^{k_1}\sqrt{1-q^{4(k_2+1)}}\ket{k_1,k_2+1} \\
\pi(x_4)\ket{k_1,k_2} &=\sqrt{1-q^{2(k_1+1)}}\ket{k_1+1,k_2} .
\end{align*}

\subsection{The  instanton bundle}\label{sec:4.1}

The algebra  $\mathcal{O}(S^7_q)$ carries a coaction of the 
Hopf *-algebra $\mathcal{O}(SU_q(2))$, which makes it a faithfully flat Hopf--Galois extension of its subalgebra of coinvariant elements \cite{LPR06}.
We denote by
 $\alpha,\gamma$ the generators
 of the Hopf *-algebra $\mathcal{O}(SU_q(2))$, with  $\alpha^*,\gamma^*$ their   adjoints, and
 relations that can be obtained by imposing that the matrix
\begin{equation}\label{eq:T1}
\tt{1}:=\mat{c}{\alpha & -q\gamma^* \\ \gamma & \alpha^*}
\end{equation}
is a unitary corepresentation. 
We arrange the generators of $\mathcal{O}(S^7_q)$ in the matrix
\begin{equation}\label{eq:Psi}
u:=\left(\begin{array}{rr}
q x_1 & q x_2 \\
-q^2 x_2^* & q^3 x_1^* \\
-x_3 & x_4   \\
x_4^* &  q x_3^*
\end{array}\right) \; .
\end{equation}
The coaction of $\mathcal{O}(SU_q(2))$ on $\mathcal{O}(S^7_q)$ is written on the generators as
\begin{equation}\label{eq:Coa0}
\delta_R(u_i^k) = \sum_{j=0,1}u_i^j \otimes (\tt{1})_j^k
\end{equation}
for all $i\in\{1, \dots, 4\}$ and $k\in\{ 0,1\}$,
and extended as an algebra map to the whole algebra 
$\mathcal{O}(S^7_q)$. Note that we count the rows and columns of the matrix 
\eqref{eq:T1} and the columns of the matrix \eqref{eq:Psi} starting from $0$.
The formula \eqref{eq:Coa0} can be written in compact form as
\begin{equation}\label{eq:Coa}
\delta_R(u)=u\,\dot{\ot}\,\tt{1} .
\end{equation}
The  subalgebra of elements of $\mathcal{O}(S^7_q)$ which are coinvariant for the coaction, that is those $b \in \mathcal{O}(S^7_q)$ such that $\delta_R(b)=b \ot 1$,
is denoted $\mathcal{O}(S^4_q)$.
The inclusion
\[
\mathcal{O}(S^4_q)\subset\mathcal{O}(S^7_q)
\]
is a faithfully flat $\mathcal{O}(SU_q(2))$--Galois extension. Geometrically this is a quantum principal bundle on the quantum $4$-sphere $S^4_q$ with structure quantum group $SU_q(2)$ and total space $S^7_q$. 
Explicitly the *-algebra $\mathcal{O}(S^4_q)$ is generated by the entries of the matrix
\begin{equation}\label{eq:P}
p:=u u^*= \left(\begin{array}{cccc}
 q^{-2}   y_0 & 0 & y_1 & y_2 \\
    0 &   y_0 & q^{-2} y_2^* & -q^2 y_1^t* \\
    y_1^* & q^{-2} y_2 & 1- q^{-4} y_0 & 0 \\
    y_2^* & -q^2 y_1 & 0 & 1- q^2 y_0
\end{array}\right) \;,
\end{equation}
where 
$$
y_0 := q^4(x_1 x_1^*+ x_2 x_2^*) \qquad
y_1 := - q x_1 x_3^* + q x_2 x_4^* \qquad
y_2 := q x_1 x_4 + q^2 x_2 x_3 \;.
$$
The relations among generators are encoded in the equality $p^2=p$, which follows from $u^*u=1_2$. There are commutation relations 
\begin{align}\label{com-relS4}
& y_1 y_2 = q^4 y_2 y_1  \, ,\qquad   \quad 
y_1^* y_2 = y_2 y_1^* \, ,\quad 
&& y_0 y_1 = q^{-2} y_1 y_0 \, ,\qquad  \quad 
y_0 y_2 = q^4 y_2 y_0  \, ,\nonumber
\\
& y_1 y_1^* -q^4 y_1^* y_1 = (q^{-2}-1)y_0  \, ,
&& y_2 y_2^*  - q^{-4} y_2^* y_2 =(1 -q^{-4}) y_0^2\, ,
\end{align}
and the sphere relation
$$
q^4 y_1^* y_1  + q^{-4} y_2^* y_2  =  y_0(1-  y_0)
\quad , \quad
\big( \mbox{ or \quad} y_1 y_1^*  + y_2 y_2^*  =  q^{-2} y_0(1- q^{-2} y_0) \big)
$$
 validating the interpretation as a quantum $4$-sphere.

The projection $p$ in \eqref{eq:P}  determines a class $[\mathscr{E}] $ in the $K$-theory of  $\mathcal{O}(S^4_q)$.  
As shown in \cite{LPR06}, the projection $p$ describes the quantum vector bundle on the  $4$-sphere $S^4_q$  associated with the Hopf--Galois extension $\mathcal{O}(S^4_q)\subset\mathcal{O}(S^7_q)$ via the fundamental corepresentation $\tt{1}$ of $\mathcal{O}(SU_q(2))$.

\subsection{The K-theory ring}\label{sec:4.2}

The C*-enveloping algebra of $\mathcal{O}(S^7_q)$ is isomorphic to the one of the Vaksman-Soibelman quantum $7$-sphere \cite{Sau17}, here denoted by $C(S^7_q)$. We let $C(S^4_q)$ be the C*-enveloping algebra of $\mathcal{O}(S^4_q)$, not the closure in $C(S^7_q)$.
This is isomorphic to the minimal unitalization of the compact operators (see \cite{LPR06}), which in turn is isomorphic to the C*-algebra $C(S^2_q)$ of the standard Podle\'s sphere, as proved in \cite{She91}.

\begin{lemma}
$K_0\big(C(S^7_q)\big)\cong\Z$ is generated by the class of the unit.
\end{lemma}

\begin{proof}
As mentioned, by \cite{Sau17}  the algebra $C(S^7_q)$ is isomorphic to the C*-algebra of a Vaksman-Soibelman sphere, whose K-theory is \mbox{$K_0\big(C(S^7_q)\big)\cong\Z$}, as computed in \cite{VS91}. From \cite[Sec.~4.3]{HL04} the K-group of these odd dimensional quantum spheres is generated by the class of the unit.
\end{proof}
We know from  \cite{LPR06} that
the group $K_0\big(C(S^4_q)\big)$ is isomorphic to $\Z^2$.

When restricted to the subalgebra $\mathcal{O}(S^4_q)$,
the character $\varepsilon:\mathcal{O}(S^7_q)\to\C$ reduces to the trivial representation, while $\pi$ reads
\begin{align}
\pi(y_0)\ket{k_1,k_2} &=q^{4+2 k_1+ 4 k_2}\ket{k_1,k_2} \\
\pi(y_1)\ket{k_1,k_2} &= q^{k_1 +2 k_2}\sqrt{1-q^{2 k_1}}\ket{k_1-1,k_2} \label{repc} \\
\pi(y_2)\ket{k_1,k_2} &= q^{2(k_1 +k_2+2)}\sqrt{1-q^{4(k_2+1)}}\ket{k_1,k_2 +1} \; .
\end{align}
These are the representations $\beta$ and $\sigma$   on 
$\mathcal{O}(S^4_q)$ constructed in \cite{LPR06} (and used there to compute  the instanton number of the quantum vector bundle given by the projection $p$).

\begin{lemma}
For all $b\in\mathcal{O}(S^4_q)$, the operator $(\pi-\varepsilon)(b)$ is of trace class.
\end{lemma}

\begin{proof}
It is enough to observe that $(\pi-\varepsilon)(y_j)$ is of trace class for all $j=0,1,2$.
\end{proof}

Let us denote $\mathscr{E} = p\, (\mathcal{O}(S^4_q))^4$ the right $\mathcal{O}(S^4_q))$-module 
of  sections of the vector bundle of $p=u u^*$.
From Proposition \ref{prop:undertheassumption}, since $u$ has $2$ columns, $\mathrm{ch}_0(\mathscr{E})=\varepsilon(\tr\,p)=2$.
Moreover
$$
(\pi-\varepsilon)(\tr\, p)= q^{-2}(1+q^2-q^{-2}-q^4)\pi(y_0).
$$
Since $\sum_{k_1,k_2\in\N}q^{2k_1+4k_2}=1/(1-q^2)(1-q^4)$, we get
\begin{equation}\label{eq:wecomputed}
\mathrm{ch}_1(\mathscr{E})=\tr_{\mathcal{H}}\big(\pi(\tr\,p)-\varepsilon(\tr\,p)\big)=\frac{q^2+q^4-1-q^6}{(1-q^2)(1-q^4)}=-1 .
\end{equation}
The general theory then applies:

\begin{cor}\label{cor:54}
$K_0\big(C(S^4_q)\big)$ is a free $\Z$-module generated by $[1]$ and $[\mathscr{E}]$, and we have the relation
\begin{equation}\label{sofE}
4-4[\mathscr{E}]+[\mathscr{E}\ot_{\mathcal{O}(S^4_q)}\mathscr{E}]=0 .
\end{equation}
\end{cor}

As a further application of the general theory we compute the characteristic classes of modules associated to irreducible corepresentations of $H:=\mathcal{O}(SU_q(2))$.
For the time being, we only need the fact that irreducible corepresentations are labelled by $n\in\N$. Let $V_n$ be the vector space underlying the $n+1$ dimensional irreducible corepresentation, and call
\begin{equation}\label{eq:En}
\mathscr{E}_n:=\mathcal{O}(S^7_q)\;\Box^{H}\,V_n .
\end{equation}
the associated $\mathcal{O}(S^4_q)$-bimodule. In particular, $\mathscr{E}_1\cong\mathscr{E}$ is the right module in Corollary~\ref{cor:54}.

\begin{prop}\label{prop:55}
For every $n\geq 1$,
\beq\label{ch1n}
\mathrm{ch}_1([\mathscr{E}_n])=-\frac{1}{6}n(n+1)(n+2) .
\eeq
\end{prop}
\begin{proof}
From the known decomposition $V_1\otimes V_n\cong V_{n+1}\oplus V_{n-1}$, using \eqref{eq:monoidal-1} and \eqref{eq:monoidal-2}, one gets the bimodule isomorphism
\begin{equation}\label{eq:casospeciale}
\mathscr{E}_1\otimes_{\mathcal{O}(S^4_q)}\mathscr{E}_n\cong\mathscr{E}_{n+1}\oplus\mathscr{E}_{n-1} .
\end{equation}
From \eqref{eq:provaB} we get
\[
\mathrm{ch}_1([\mathscr{E}_{n+1}]) =\mathrm{ch}_0([\mathscr{E}_1])\mathrm{ch}_1([\mathscr{E}_n])+\mathrm{ch}_0([\mathscr{E}_n])\mathrm{ch}_1([\mathscr{E}_1])-\mathrm{ch}_1([\mathscr{E}_{n-1}]) .
\]
Using $\mathrm{ch}_1([\mathscr{E}_1])=-1$, as computed in \eqref{eq:wecomputed}, and $\mathrm{ch}_0([\mathscr{E}_n])=\dim(V_n)=n+1$, we find
\[
\mathrm{ch}_1([\mathscr{E}_{n+1}])=2\mathrm{ch}_1([\mathscr{E}_n])-n-1-\mathrm{ch}_1([\mathscr{E}_{n-1}])  .
\]
Formula \eqref{ch1n} follows by induction on $n\geq 1$.
\end{proof}

\begin{rem}\label{rem:detq}
A special case of \eqref{eq:casospeciale} is the isomorphism $
\mathscr{E}_1\otimes_{\mathcal{O}(S^4_q)}\mathscr{E}_1\cong\mathscr{E}_2\oplus\mathscr{E}_0$,
where the free bimodule $\mathscr{E}_0$ is the analogue of the determinant line bundle of $\mathscr{E}_1$. With an abuse of notations we can denote it by $\mathscr{E}_1\wedge_{\mathcal{O}(S^4_q)}\mathscr{E}_1$.
Then, the relation \eqref{sofE} can be interpreted as the vanishing of the ``square'' of the Euler class in K-theory of the instanton bundle on $S^4_q$ given by:
\[
\chi(\mathscr{E}_1)= 1 - [\mathscr{E}_1] + [\mathscr{E}_1\wedge_{\mathcal{O}(S^4_q)}\mathscr{E}_1] = 2 - [\mathscr{E}_1] .
\]
This parallel the classical result in Example \ref{ex:S4}. 
\end{rem}

\subsection{Projections from corepresentations of $SU_q(2)$}\label{sec:4.4}

In this section we construct a trivializing pairs and then projections describing the vector bundles on the quantum $4$-sphere $S^4_q$ associated to finite-dimensional irreducible corepresentations of the Hopf algebra $\mathcal{O}(SU_q(2))$, generalizing the construction done in \cite{LPR06} for the fundamental one given by \eqref{eq:P} above.

We need the (rescaled) $q$-number and $q$-binomial coefficients defined by \cite[Sect.~2.1.2]{KS97}
\begin{align*}
\qqnum{n}:=\qqnum{n}_{q^2}:=\frac{1-q^{2n}}{1 - q^{2}} , \qquad\quad
\qbin{n}{k}:=\qbin{n}{k}_{q^2}:=
\frac{\prod_{i=k+1}^{n} (1-q^{2i})}{\prod_{j=1}^{n-k} (1-q^{2j}) } ,
\end{align*}
for all $n\in\N$ and $0\leq k\leq n$, with the convention that an empty product is equal to $1$. The following property will be useful in the computations:
\beq\label{pro-qqn}
q^{2k}\qqnum{n-k}+\qqnum{k}=\qqnum{n} .
\eeq

Recall that, like in the classical case, for every $n\in\N$ there is an irreducible corepresentation of $\mathcal{O}(SU_q(2))$ on a complex vector space $V_n$ of dimension $n+1$ given on a basis by a unitary matrix $\tt{n}\in M_{n+1}\big(\mathcal{O}(SU_q(2))\big)$. The general formula for $\tt{n}$ is for instance in \cite[Eq.~(3.11)]{Koo89} (with slighly different notations:
$(\tt{n})_i^j$ is $t^\ell_{i-\ell,j-\ell}$ in \cite{Koo89}, with $\ell=n/2$). Thus, for example, $\tt{1}$ is the matrix \eqref{eq:T1}, and
\[
\tt{2}=
\begin{pmatrix}
\alpha^2 & - (1+ q^{2})^{\half}  \alpha \gamma^* & q^2 
  (\gamma^* )^2
\\
(1+ q^{2})^{\half}  \gamma \alpha  & 1-(1+q^2) \gamma \gamma^* & 
- q (1+ q^{2})^{\half} \alpha^* \gamma^* 
\\
 \gamma^2  &  q^{-1}(1+ q^{2})^{\half} \gamma \alpha^* & (\alpha^* )^2
\end{pmatrix} \,.
\]
Note again that we count rows and columns starting from $0$.

\begin{prop}\label{prop:ne2}
Let $\uu{2}\in M_{4^2\times 3}\big(\mathcal{O}(S^7_q)\big)$ be the matrix with row $I$, labelled by a multi-index $I=(i,j)\in\{ 1, \dots, 4\}^2$, given by
\begin{equation}\label{U2}
(\uu{2})_I :=
 \left(u_i^0 u_j^0\, , \; 
(1+ q^{2})^{- \half}   ( q u_i^0 u_j^1+ u_i^1 u_j^0)\, , \;  
u_i^1 u_j^1 \right) .
\end{equation}
Then, this matrix is a partial isometry and transforms under the coaction $\delta_R$ with the matrix $\tt{2}$, that is:
\[
(\uu{2})^*\,\uu{2}=1_3 ,\qquad\quad
\delta_R(\uu{2})=\uu{2}\,\dot{\otimes}\,\tt{2} .
\]
\end{prop}

\begin{proof}
The proof is a direct explicit computation using the fact that rows of $u$ are ``orthonormal'', meaning that $(u^j)^*\,u^k=\delta^{jk}$, and the fact that $\delta_R$ is an algebra morphism.
\end{proof}

It follows from the previous proposition and from Lemma \ref{lemma:44} that the matrix
\begin{equation} \label{P2}
\pp{2}:= \uu{2} (\uu{2})^*  
\end{equation}
is a projection with entries in $\mathcal{O}(S^4_q)$, and one has the isomorphism of right $\mathcal{O}(S^4_q)$-module
\[
\pp{2}\,\mathcal{O}(S^4_q)^{4^2}\cong\mathscr{E}_2 ,
\]
with $\mathscr{E}_2$ given in \eqref{eq:En}.\footnote{More precisely, $\pp{2}\,\mathcal{O}(S^4_q)^{4^2}$ is isomorphic to the module associated to the corepresentation $V_2^\vee$, but
irreducible corepresentations of $\mathcal{O}(SU_q(2))$ are self-dual.\label{footnota}}

Together with the $q$-symmetrized vector in \eqref{U2} there is also a $q$-antisymmetrized one. 
Let $\ww{0}\in\mathcal{O}(S^7_q)^{4^2}$ be the column vector with $I$-th entry
\begin{equation}\label{pro-V}
(\ww{0})_I := (1+ q^{2})^{- \half}  \left(u_i^0 u_j^1 - q u_i^1 u_j^0 \right),
 \quad I=(i,j)\in\{1,\ldots,4\}^2 .
 \end{equation}
Its elements belong to $\mathcal{O}(S^4_q)$ being proportional to the $q$-minors 
\[
m_{ij}:= u_i^0 u_j^1   - q u_i^1 u_j^0
\]
of the matrix $u$, which transform under $\delta_R$ with the quantum determinant:
\[
\delta_R(m_{ij})=m_{ij} \ot det_q \, , \qquad  \forall \, i,j=1, \dots, 4,
\]
for $det_q= \alpha \alpha^* +q^2 \gamma^* \gamma=1$ (cf.~\cite[Remark 1]{LPR06}).
In particular $m_{ii} = u_i^0 u_i^1   - q u_i^1 u_i^0 = 0$.

One has $(\ww{0})^* \uu{2}=0$ and $(\ww{0})^*\ww{0}=1$ from which we deduce that
\begin{equation}\label{proj-Q}
\qq{0}:=\ww{0}(\ww{0})^* \in M_{4^2}(\mathcal{O}(S^4_q))
\end{equation}
is a projection orthogonal to $\pp{2}$, von Neumann equivalent to the trivial idempotent $(\ww{0})^*\ww{0}=1 \in M_1(\mathcal{O}(S^4_q))$ (since the entries of the partial isometry $\ww{0}$ belong to $\mathcal{O}(S^4_q)$) and describes the trivial determinant line bundle of $\mathscr{E}_1$ (see Remark \ref{rem:detq}).

\begin{rem}\label{appA}
It is instructive to see explicitly the triviality in K-theory of the projection $\qq{0}$.
For $\eta(b):= \tr_{\mathcal{H}}\big(\pi(b)-\varepsilon(b)\big)$ for each $b \in \mathcal{O}(S^4_q)$, from
 \eqref{repc} we compute
\begin{align}\label{rep-pi-quad}
\eta(y_0)&=  (1+q^{-4}) (1+q^2) \eta(y_0^2) \nonumber \\
\eta(y_1^* y_1)&= \eta(y_1 y_1^*)=q^{-6}(1+q^4) \eta (y_0^2) , \quad
\eta(y_2^* y_2)= \eta(y_2 y_2^*)=\eta (y_0^2) 
\end{align}
with 
$\eta(y_0^2)= q^8 (1-q^4)^{-1} (1-q^8)^{-1}$. For the matrix trace we compute
\begin{align*}
 (1+q^2) \tr(\qq{0}) &= (1+q^2) -2 y_0 \left(q^2+q^{-2} \right) +  y_0^2 \left(1+q^{-2}+ q^{-6}+q^4 \right)
\\
& \quad + 2q^2 y_1 y_1^* + 2 q^4 y_1^* y_1   + (1 +q^{-6}) y_2 y_2^* + (q^{-2}+  q^4) y_2^* y_2 
\\
&= (1+q^2) + (1- q^{-2} )  
\Big( - 2 (q^2 -1) y_0  + (1 - q^{-8})(q^4 -q^2 -1 )  y_0^2 
\\ & \hspace{4.5cm}
+ (1 - q^{-8}) (q^4  +q^2 +1 ) y_2^* y_2 \Big) .
\end{align*}
It then follows that $\mathrm{ch}_0([\qq{0}])=\epsilon(\tr(\qq{0}))=1$ and, 
by using  
\eqref{rep-pi-quad}, that $\eta$ vanishes on the term in parenthesis so that 
$\mathrm{ch}_1([\qq{0}]) = 0$.
\end{rem}

We now pass to the module $\mathscr{E}_n$ associated to the general representation $V_n$, $n\geq 2$.
Given two column vectors $v \in \mathcal{O}(S^7_q) \ot \C^m$ and
$w \in \mathcal{O}(S^7_q) \ot \C^{k}$, $m,k \in N$,  we write $v \dott w$  for the vector in $\mathcal{O}(S^7_q) \ot \C^{m k}$ of components $v_i w_j$, with $i\in\{1, \dots, m\}$ and $j\in\{1, \dots, k\}$. For instance, the matrices $\uu{2}$ in \eqref{U2}   and $W$ in \eqref{pro-V} are written as 
\begin{align*}
\uu{2} &= 
 \left(u^0 \dott u^0\, , \; 
\qqnum{2}^{- \half}   ( q u^0 \dott u^1+ u^1 \dott u^0)\, , \;  
u^1 \dott u^1 \right)\; , \; \\
\ww{0} &=
\qqnum{2}^{- \half}     (u^0 \dott u^1 -q u^1\dott u^0)
\end{align*}

\begin{prop}\label{prop:58}
Let \mbox{$\uu{1}:=u=(u^0,u^1)$} be the matrix in \eqref{eq:Psi}.
For $n\geq 2$, let $\uu{n}\in M_{4^n\times (n+1)}\big(\mathcal{O}(S^7_q)\big)$ be the matrix with $k$-th column defined recursively by:
\begin{equation}\label{Un-cl}
\uu{n}^k = \qqnum{n}^{-\half}  \left( q^k \qqnum{n-k}^{\half}  u^0 \dott \uu{n-1}^k +   \qqnum{k}^{\half}  u^1 \dott \uu{n-1}^{k-1} \right) ,
\end{equation}
for $k\in\{0,1, \dots, n\}$ (with the convention that $\uu{n-1}^k:=0$ for $k=n$).
Then, this matrix is a partial isometry and transforms under the coaction $\delta_R$ with the matrix $\tt{n}$, that is:
\begin{gather}
(\uu{n})^*\,\uu{n}=1_{n+1} , \label{eq:58a} \\
\delta_R(\uu{n})=\uu{n}\,\dot{\otimes}\,\tt{n} . \label{eq:58b}
\end{gather}
\end{prop}

\begin{proof}
We prove the statements by induction on $n$. The equality \eqref{eq:58a} is obvious: if the columns of $u$ and $\uu{n-1}$ are orthonormal, the columns of $\uu{n}$ are orthonormal as well.

Concerning \eqref{eq:58b}, the base step follows from the definition \eqref{eq:Coa} of the coaction.
Assume, then, that for some $n\geq 2$, $\delta_R(\uu{n-1})=\uu{n-1}\,\dot{\otimes}\,\tt{n-1}$.
We need the explicit expression of the first column of $\tt{n}$, given by $(\tt{n})_{k}^0=\qbin{n}{k}^{\half} \gamma^k \alpha^{n-k}$.
Using $\delta_R(u^0) = u^0 \otimes \alpha + u^1 \otimes \gamma$, and using the inductive hypothesis, we get
\begin{align*}
\delta_R(\uu{n}^0) & = \delta_R(u^0 \dott \uu{n-1}^0) \\
& = \sum_{k=0}^{n-1} \Big( u^0 \dott \uu{n-1}^k \otimes \alpha \,(\tt{n-1})_k^0
+ u^1 \dott \uu{n-1}^k \otimes \gamma\, (\tt{n-1})_k^0 \Big)
\\
& = \sum_{k=0}^{n-1} \qbin{n-1}{k}^{\half} \Big( q^ku^0 \dott \uu{n-1}^k \otimes  \gamma^k \alpha^{n-k}
+ u^1 \dott \uu{n-1}^k \otimes\gamma^{k+1} \alpha^{n-1-k} \Big)
\\
& = \sum_{k=0}^{n-1} \qbin{n-1}{k}^{\half} q^ku^0 \dott \uu{n-1}^k \otimes  \gamma^k \alpha^{n-k}
+ \sum_{k=1}^{n} \qbin{n-1}{k-1}^{\half}u^1 \dott \uu{n-1}^{k-1} \otimes\gamma^{k} \alpha^{n-k}
\\
& =  \qqnum{n}^{-\half}\sum_{k=0}^{n}   \left( q^k \qqnum{n-k}^{\half}  u^0 \dott \uu{n-1}^k +   \qqnum{k}^{\half}  u^1 \dott \uu{n-1}^{k-1} \right) \otimes \qbin{n}{k}^{\half} \gamma^{k} \alpha^{n-k} .
\end{align*}
Thus,
\begin{equation}\label{eq:thisiseq}
\delta_R(\uu{n}^0)= \sum_{k=0}^{n}\uu{n}^k\otimes (\tt{n})_k^0 .
\end{equation}
This yields the statement to be proven, since applying again the coaction  we find
\begin{align*}
\sum_{k=0}^{n}\delta_R(\uu{n}^k)\otimes (\tt{n})_k^0 
&=(\delta_R\otimes\id)\delta_R(\uu{n}^0) \\
&=(\id\otimes\Delta)\delta_R(\uu{n}^0)= \sum_{j,k=0}^{n}\uu{n}^j\otimes (\tt{n})_j^k \otimes (\tt{n})_k^0\, .
\end{align*}
Since the elements $(\tt{n})_k^0$ are linearly independent, from the latter identity we deduce that $\delta_R(\uu{n}^k)=\sum_{j=0}^{n}\uu{n}^j\otimes (\tt{n})_j^k$ for all $0\leq k\leq n$.
\end{proof}

It follows from the previous proposition and from Lemma \ref{lemma:44} that the matrix
\begin{equation} \label{Pn}
\pp{n}:= \uu{n} (\uu{n})^*  
\end{equation}
is a projection with entries in $\mathcal{O}(S^4_q)$ 
and one has the isomorphism of right $\mathcal{O}(S^4_q)$-module
\[
\pp{n}\,\mathcal{O}(S^4_q)^{4^n}\cong\mathscr{E}_n ,
\]
with $\mathscr{E}_n$ given in \eqref{eq:En} (see again the footnote at page \pageref{footnota}).

We obtain a different family of projections by a kind of partial $q$-antisymmetrization procedur similar to \eqref{pro-V}.

\begin{prop}
For $n\geq 0$, let $\ww{n}\in M_{4^{n+2}\times (n+1)}\big(\mathcal{O}(S^7_q)\big)$ be the matrix with $k$-th column given by:
\begin{equation}\label{Wn-cl}
\ww{n}^k = \qqnum{n+2}^{-\half}  \left( \qqnum{k+1}^{\half}  u^0 \dott \uu{n+1}^{k+1} - q^{k+1}\qqnum{n+1-k}^{\half}  u^1 \dott \uu{n+1}^k \right) ,
\end{equation}
for $k\in\{0,1, \dots, n\}$. Then,
\begin{align}
(\ww{n})^*\,\ww{n} &=1_{n+1} , \label{eq:59a} \\
(\ww{n-1})^*\,\uu{n+1} &=0_n , \qquad (\text{for }n\geq 1) \label{eq:59b} \\
\delta_R(\ww{n})={} & \ww{n}\,\dot{\otimes}\,\tt{n} . \label{eq:59c}
\end{align}
\end{prop}

\begin{proof}\noeqref{eq:59a}%
The first two identities are a consequence of orthonormality of the columns of $u$ and $\uu{n+1}$.
For \eqref{eq:59c}, similarly to the proof of Proposition~\ref{prop:58}, it is enough to prove it for the $0$-th column, that is
\begin{equation}\label{eq:Wind}
\delta_R(\ww{n}^0)= \sum_{k=0}^{n}\ww{n}^k\otimes (\tt{n})_k^0 .
\end{equation}
We prove this by induction on $n\geq 0$.

From \eqref{Wn-cl}, one gets the recursive formula for $\ww{n}^k$:
\begin{multline}\label{rec-formula-Vk}
\qqnum{n+2}^{\half} \ww{n}^k =q^{k+1} \qqnum{n-k}^{\half}  
u^0 \dott \ww{n-1}^k +q \qqnum{k}^{\half}  \, u^1 \dott \ww{n-1}^{k-1}  \\ \qquad +  \qqnum{n+1}^{\half}(1+q^2)^{\half} {\ww{0}} \dott \uu{n}^{k}  
\end{multline}
In particular, for $k=0$ we get
\beq\label{rec-Vk0}
\qqnum{n+2}^{\half} \ww{n}^0 = q \qqnum{n}^{\half} u^0 \dott \ww{n-1}^0 + 
\qqnum{n+1}^{\half}(1+q^2)^{\half}  \ww{0} \dott \uu{n}^0 . 
\eeq
Assume that \eqref{eq:Wind} holds for $\ww{n-1}$, for some $n\geq 1$.
Applying $\delta_R$ to \eqref{rec-Vk0}, using $\delta_R(u^0) = u^0 \otimes \alpha + u^1 \otimes \gamma$ and
$\alpha\gamma=q\gamma\alpha$, using the explicit expression of 
$(\tt{n})_k^0$, \eqref{eq:thisiseq}, and the inductive hypothesis we prove that \eqref{eq:Wind} holds for $\ww{n}$.
\end{proof}

It follows from the previous proposition and from Lemma \ref{lemma:44} that the matrix
\begin{equation} \label{Qn}
\qq{n}:= \ww{n} (\ww{n})^*  
\end{equation}
is a projection with entries in $\mathcal{O}(S^4_q)$, and one has the isomorphism of right $\mathcal{O}(S^4_q)$-module
\[
\qq{n}\,\mathcal{O}(S^4_q)^{4^{n+2}}\cong\mathscr{E}_n .
\]
Thus, $\pp{n}$ and $\qq{n}$ are von Neumann equivalent. The explicit equivalence is realized by the partial isometry $\uu{n}(\ww{n})^*$, whose entries are coinvariant, hence in $\mathcal{O}(S^4_q)$.

From \eqref{eq:59b} we see that $\pp{n+1}$ and $\qq{n-1}$ are orthogonal, for all $n\geq 1$. So, their sum
$\pp{n+1}+\qq{n-1}$ is also a projection.
We already know from the corepresentation theory of $SU_q(2)$, and being the cotensor product an additive monoidal functor, that
\[
(\pp{n+1}+\qq{n-1})\mathcal{O}(S^4_q)^{4^{n+1}}\cong\mathscr{E}_{n+1} \oplus\mathscr{E}_{n-1}\cong\mathscr{E}_1\otimes_{\mathcal{O}(S^4_q)}\mathscr{E}_n .
\]
(In fact, we already used the second isomorphism to prove Proposition \ref{prop:55}.) 

The isomorphism between the first and third module can be proved explicitly, without using the general theory of cotensor products, as shown below.

\begin{prop}
For all $n\geq 1$, one has the isomorphism of right $\mathcal{O}(S^4_q)$-modules
\[
(\pp{n+1}+\qq{n-1})\mathcal{O}(S^4_q)^{4^{n+1}}\cong \mathscr{E}_1\otimes_{\mathcal{O}(S^4_q)}\mathscr{E}_n .
\]
\end{prop}
\begin{proof}
We have to show that the projection $\pp{n,1}$ for the module $\mathscr{E}_1\otimes_{\mathcal{O}(S^4_q)}\mathscr{E}_n$ given as in \eqref{eq:capitalP} is equal to $(\pp{n+1}+\qq{n-1})$ as matrix. From 
\eqref{Un-cl} and \eqref{Wn-cl} we obtain (in a compact form) the partial isometries 
$\uu{n,1} = (u^0 \dott \uu{n}, u^1 \dott \uu{n})$ of $\pp{n,1}$ as in Proposition~\ref{prop:UV}:
\begin{align*}
u^0 \dott \uu{n}^k &= \qqnum{n+1}^{-\half} \big( q^k \qqnum{n+1-k}^{\half} \uu{n+1}^k + 
\qqnum{k}^{\half} \ww{n-1}^{k-1} \big) \\
 u^1 \dott \uu{n}^k &=\qqnum{n+1}^{-\half}   \big( \qqnum{k+1}^{\half} \uu{n+1}^{k+1} - 
q^{k+1} \qqnum{n-k}^{\half} \ww{n-1}^k \big) .
\end{align*}
We then compute
\begin{align*}
\pp{n,1} & = \uu{n,1} (\uu{n,1})^* \\
& = \sum_{k=0}^n \big(u^0 \dott \uu{n}^k (\uu{n}^k)^* \dott (u^0)^* + 
u^1 \dott \uu{n}^k (\uu{n}^k)^* \dott (u^1)^* \big) \\
& = \qqnum{n+1}^{-1} \, \sum_{k=0}^n \big(q^{2k} \qqnum{n+1-k} \uu{n+1}^k (\uu{n+1}^k)^* + 
\qqnum{k+1} \uu{n+1}^{k+1} (\uu{n+1}^{k+1})^* \big) \\ 
&  
+ \qqnum{n+1}^{-1} \, \sum_{k=0}^n \big( \qqnum{k} \ww{n-1}^{k-1} (\ww{n-1}^{k-1})^* + 
q^{2(k+1)} \qqnum{n-k} \ww{n-1}^k  (\ww{n-1}^k)^* \big) \\
\intertext{(since the mixed terms $\uu{n+1} (\ww{n-1})^*$ and $\ww{n-1} (\uu{n+1})^*$ cancel)}
&= \qqnum{n+1}^{-1} \, \sum_{k=0}^n \Big(q^{2k} \qqnum{n+1-k} + \qqnum{k} \Big) \uu{n+1}^k (\uu{n+1}^k)^* \\ 
&  
+ \qqnum{n+1}^{-1} \, \sum_{k=0}^n \Big( \qqnum{k+1} + q^{2(k+1)} \qqnum{n-k}  \Big) \ww{n-1}^k  (\ww{n-1}^k)^* \\
&=\pp{n+1}+\qq{n-1}
\end{align*}
using for the last equality the identity in \eqref{pro-qqn}.
\end{proof}

As a corollary, using the isomorphisms
$\pp{n+1}\mathcal{O}(S^4_q)^{4^{n+1}}\cong \mathscr{E}_{n+1}$,
$\qq{n-1}\mathcal{O}(S^4_q)^{4^{n+1}}\cong \mathscr{E}_{n-1}$
and
$(\pp{n+1}+\qq{n-1})\mathcal{O}(S^4_q)^{4^{n+1}}\cong \mathscr{E}_1\otimes_{\mathcal{O}(S^4_q)}\mathscr{E}_n$,
we obtain an independent proof of the recursive formula
\[
\mathrm{ch}_1([\mathscr{E}_{n+1}]) = \mathrm{ch}_1([\mathscr{E}_1\otimes_{\mathcal{O}(S^4_q)}\mathscr{E}_n]) -\mathrm{ch}_1([\mathscr{E}_{n-1}]) ,
\]
which using \eqref{eq:provaB} gives for the Chern number
\[
\mathrm{ch}_1([\mathscr{E}_n]) = \frac{1}{6} n (n+1) (n+2) \, \mathrm{ch}_1([\mathscr{E}_1]) ,
\]
consistently with Proposition~\ref{prop:55}.

\begin{center}
\textsc{Acknowledgements}
\end{center}

\noindent
FD, GL and CP are members of INdAM-GNSAGA.
FD and CP acknowledges support by the University of Naples Federico II under the grant FRA 2022 \emph{GALAQ: Geometric and ALgebraic Aspects of Quantization}.
GL acknowledges support from PNRR MUR projects PE0000023-NQSTI.

\appendix
\section{$\Z$-modules}\label{sec:app}
To see what theorems carry over from vector spaces to free $\Z$-modules, one can see e.g.~Chap.~12 in Artin's book \textit{Algebra} \cite{Art91}. In particular: any basis of a rank $r$ free module has cardinality $r$; a set of elements of $M$ is a basis if and only if the components of the elements in any other basis are the columns of an invertible matrix (the change of basis matrix); if $M$ is free with rank $r$ its dual $M^\vee:=\mathrm{Hom}_{\Z}(M,\Z)$ is also free with the same rank $r$; and there exists a pair of dual basis of $M$ and $M^\vee$.

\begin{lemma}\label{lemma:1}
Let $M$ be a free $\Z$-module of rank $r$, $p_1,\ldots,p_r\in M$, $\varphi^1,\ldots,\varphi^r\in M^\vee$ and let
$a=(a_i^j)$ be the matrix with entries $a_i^j:=\varphi^j(p_i)$. Then, the following two conditions are equivalent
\begin{enumerate}
\item\label{en:1} $(p_1,\ldots,p_r)$ is a basis of $M$ and $(\varphi^1,\ldots,\varphi^r)$ is a basis of $M^\vee$;

\item\label{en:2} $a\in GL_r(\Z)$ (i.e.~$\det a=\pm 1$).

\end{enumerate}
\end{lemma}

\begin{proof}
Let $(e_1,\ldots,e_r)$ and $(e^1,\ldots,e^r)$ be a pair of dual basis of $M$ and $M^\vee$. Write
$p_i=\sum_{k=1}^r b_i^ke_k$ and $\varphi^j=\sum_{l=1}^r c_l^je^l$ where $b=(b_i^k)$ and $c=(c_l^j)$ are integer matrices. Condition \ref{en:1} is equivalent to the condition that both $b$ and $c$ belong to $GL_r(\Z)$.

Note that
\[
a_i^j=\sum_{k,l=1}^rc_l^je^l(b_i^ke_k)=\sum_{k=1}^rc_k^jb_i^k ,
\]
that is $a=bc$. Since $\det a=(\det b)(\det c)$, clearly $a\in GL_r(\Z)$ (i.e.~\ref{en:2} hold) if and only if both $b$ and $c$
belong to $GL_r(\Z)$ (i.e.~\ref{en:1} hold).
\end{proof}

We are interested in the $\Z$-module $K_*(A)$, where $A$ is a unital C*-algebra.
Assume that $A$ is in the UCT class (see \cite{BBWW20} for the state of the art about UCT)
and that $K_*(A)$ is torsion-free.
Using \cite[Thm.~23.1.1]{Bla86} with $B=\C$, since $K_*(A)$ is free, we get a degree $0$ isomorphism
\[
KK^*(A,\C)\cong\mathrm{Hom}_{\Z}(K_*(A),K_*(\C)) .
\]
The left hand side is $K^*(A)$, and since $K_*(\C)=K_0(\C)=\Z$ we get
\[
K^*(A)\cong K_*(A)^\vee .
\]
 
\section{Proof of Proposition~\protect{\ref{prop:sfereclassiche}}}\label{AppB}

That the matrix $p$ in \eqref{eq:projs2n} is a projection follows from the identities
$\gamma_i^*=\gamma_i$, $\gamma_i^2=1$ and $\gamma_i\gamma_j=-\gamma_j\gamma_i$ for all $i\neq j$,
and from $\sum_{i=1}^{2n+1}(x_i)^2=1$ on $S^{2n}$.
Since the $\gamma$-matrices are traceless, $\tr(p)$ is half the trace of the $2^n\times 2^n$ identity matrix.
Thus, the rank of $E$ is $\mathrm{ch}_0(E)=2^{n-1}$.
The top Chern character, the only other not vanishing character, is
\[
\mathrm{ch}_n(E)=\frac{1}{n!}\tr\left(\frac{\sqrt{-1}}{2\pi}\,\Omega\right)^{\hspace*{-4pt}n} .
\]
where
\[
\Omega=p(dp\wedge dp)p
\]
is the curvature of the Grassmannian connection defined by $p$. Here $\wedge$ is the exterior product of forms composed with the matrix product. Explicitly,
\[
\Omega_i^j=\sum_{k,l,m}p_i^k(dp_k^l\wedge dp_l^m)p_m^j ,
\]
and in particular $dp\wedge dp$ is not zero (despite the wedge product being graded-symmetric).
From $p^2=p$ and the Leibniz rule we get $(dp)p=d(p^2)-pdp=(1-p)dp$, so that $p$ commutes with $dp\wedge dp$ and then
\[
\mathrm{ch}_n(E)=\frac{1}{n! }\left(\frac{\sqrt{-1}}{2\pi}\right)^{\hspace*{-4pt}n}\tr\big(p(dp)^{2n}\big) .
\]
Using Stoke's theorem we then find
\[
\int_{S^{2n}}\mathrm{ch}_n(E)=
\int_{B^{2n+1}}d\big(\mathrm{ch}_n(E)\big)=
\frac{1}{n! }\left(\frac{\sqrt{-1}}{2\pi}\right)^{\hspace*{-4pt}n}\int_{B^{2n+1}}\tr\big((dp)^{2n+1}\big) ,
\]
where the integral on the right hand side is on the unit ball in $\R^{2n+1}$ and we extended $p$ to a function on $\R^{2n+1}$ in the obvious way, by changing the domain of the Cartesian coordinates from the sphere to the whole $\R^{2n+1}$.

We now compute the trace:
\[
\tr\big((dp)^{2n+1}\big)=\frac{1}{2^{2n+1}}\sum_{i_1,\ldots,i_{2n+1}}\tr\big(
\gamma_{i_1}\gamma_{i_2}\cdots\gamma_{i_{2n+1}}
\big)dx_{i_1}\wedge dx_{i_2}\wedge\ldots\wedge dx_{i_{2n+1}} .
\]
In the product of 1-forms, only terms with all the indices different are non-zero, i.e.~with $(i_1,\ldots,i_{2n+1})$ a permutation of $(1,\ldots,2n+1)$.
Since both gamma matrices and 1-forms anticommute, the summand is invariant under permutation of the indices and
\begin{equation}\label{eq:t}
\tr\big((dp)^{2n+1}\big)=\frac{(2n+1)!}{2^{2n+1}}\tr\big(
\gamma_{1}\gamma_{2}\cdots\gamma_{2n+1}
\big)dx_{1}\wedge dx_{2}\wedge\ldots\wedge dx_{2n+1} .
\end{equation}
Note that
\[
\frac{(2n+1)!}{2^{2n+1}n!}=(n+\tfrac{1}{2})(n-\tfrac{1}{2})(n-\tfrac{3}{2})\cdots\tfrac{1}{2}=\frac{\Gamma(n+\tfrac{3}{2})}{\sqrt{\pi}}=\frac{\pi^n}{\mathrm{Vol}(B^{2n+1})} ,
\]
where $\mathrm{Vol}(B^{2n+1})$ is the volume of the $(2n+1)$-dimensional unit ball.

The trace in \eqref{eq:t} is easily computed. Firstly,
\[
\gamma_{2i+1}\gamma_{2i+2}=1^{\otimes i}\otimes\sigma_1\sigma_2\otimes 1^{\otimes n-i-1}
=\sqrt{-1}\,(1^{\otimes i}\otimes\sigma_3\otimes 1^{\otimes n-i-1})
\]
for all $0\leq i\leq n-1$, so that
\[
\gamma_{1}\gamma_{2}\cdots\gamma_{2n}=\big(\sqrt{-1}\big)^n\gamma_{2n+1}.
\]
Using $\tr(\gamma_{2n+1}^2)=\tr(1_N)=N=2^n$, we find
\[
\tr\big((dp)^{2n+1}\big)=\frac{(2\pi)^n n!}{\mathrm{Vol}(B^{2n+1})}\big(\sqrt{-1}\big)^n dx_{1}\wedge dx_{2}\wedge\ldots\wedge dx_{2n+1} .
\]
Thus,
\[
\int_{S^{2n}}\mathrm{ch}_n(E)=\frac{(-1)^n}{\mathrm{Vol}(B^{2n+1})}\int_{B^{2n+1}}dx_{1}\wedge dx_{2}\wedge\ldots\wedge dx_{2n+1}=(-1)^n.
\]
The last claim of Proposition \ref{prop:sfereclassiche} follows from Proposition \ref{prop:next}.

\end{document}